\documentclass[UTF-8,reqno]{amsart}
\usepackage{enumerate}
\usepackage{mhequ}
\usepackage[margin=1.2in]{geometry}
\usepackage{amssymb,url,color, booktabs,nccmath}
\usepackage{mathrsfs}
\usepackage{enumitem}
\usepackage{graphicx}
\usepackage{tikz}
\usepackage{amsthm}
\usepackage{amsmath}
\usepackage{amssymb}

\usepackage{color}
\usepackage[colorlinks=true]{hyperref}
\hypersetup{
	linkcolor=blue,          
	citecolor=red,        
	filecolor=blue,      
	urlcolor=cyan
}

\newcommand{\ud}{\,\mathrm{d}}
\newcommand{\udiv}{\, \mathrm{div}}
\newcommand{\eps}{\varepsilon}

\definecolor{darkergreen}{rgb}{0.0, 0.5, 0.0}

\numberwithin{equation}{section}

\newcommand{\be}{\begin{eqnarray}}
	\newcommand{\ee}{\end{eqnarray}}
\newcommand{\ce}{\begin{eqnarray*}}
	\newcommand{\de}{\end{eqnarray*}}
\newtheorem{theorem}{Theorem}[section]
\newtheorem{lemma}[theorem]{Lemma}
\newtheorem{remark}[theorem]{Remark}
\newtheorem{definition}[theorem]{Definition}
\newtheorem{proposition}[theorem]{Proposition}
\newtheorem{Examples}[theorem]{Example}
\newtheorem{corollary}[theorem]{Corollary}

\newtheorem*{theorem*}{Theorem}
\newtheorem*{remark*}{Remark}

\newcommand{\assign}{:=}

\newcommand{\mathd}{\mathrm{d}}

\newcommand{\tmop}[1]{\ensuremath{\operatorname{#1}}}

\def\eps{\varepsilon}

\def\[{{\Big[}}
\def\]{{\Big]}}
\def\<{{\langle}}
\def\>{{\rangle}}
\def\({{\Big(}}
\def\){{\Big)}}

\def\bx{{\mathbf{x}}}

\def\dif{{\mathord{{\rm d}}}}

\def\min{{\mathord{{\rm min}}}}

\def\={&\!\!=\!\!&}

\def\cM{{\mathcal M}}
\def\cN{{\mathcal N}}

\def\cP{{\mathcal P}}

\def\cX{{\mathcal X}}

\def\mH{{\mathbb H}}

\def\mK{{\mathbb K}}

\def\mP{{\mathbb P}}

\def\mR{{\mathbb R}}

\def\mW{{\mathbb W}}

\def\1{{\mathbf{1}}}

\def\sB{{\mathscr B}}

\def\geq{\geqslant}
\def\leq{\leqslant}

\def\le{\leqslant}

\def\div{\mathord{{\rm div}}}

\def\eps{\varepsilon}

\def\[{{\Big[}}
\def\]{{\Big]}}
\def\<{{\langle}}
\def\>{{\rangle}}
\def\({{\Big(}}
\def\){{\Big)}}

\def\bx{{\mathbf{x}}}

\def\dif{{\mathord{{\rm d}}}}

\def\min{{\mathord{{\rm min}}}}

\def\={&\!\!=\!\!&}
\def\bt{\begin{theorem}}
	\def\et{\end{theorem}}
\def\bl{\begin{lemma}}
	\def\el{\end{lemma}}
\def\br{\begin{remark}}
	\def\er{\end{remark}}
\def\bx{\begin{Examples}}
	\def\ex{\end{Examples}}
\def\bd{\begin{definition}}
	\def\ed{\end{definition}}
\def\bp{\begin{proposition}}
	\def\ep{\end{proposition}}
\def\bc{\begin{corollary}}
	\def\ec{\end{corollary}}

\def\geq{\geqslant}
\def\leq{\leqslant}

\def\le{\leqslant}

\def\div{\mathord{{\rm div}}}

\def\<{\langle} \def\>{\rangle}

\allowdisplaybreaks

\begin{document}

	\title[Non-exchangeable interacting diffusions  with singular kernels ]{Mean-field limit of Non-exchangeable interacting diffusions with singular kernels}
	
	\author{Zhenfu Wang}
	\address[Z. Wang]{Beijing International Center for Mathematical Research, Peking University, Beijing 100871, China}
	\email{zwang@bicmr.pku.edu.cn}
	
	\author{Xianliang Zhao}
	\address[X. Zhao]{ Academy of Mathematics and Systems Science,
		Chinese Academy of Sciences, Beijing 100190, China; Fakult\"at f\"ur Mathematik, Universit\"at Bielefeld, D-33501 Bielefeld, Germany}
	\email{xzhao@math.uni-bielefeld.de}
	
	\author{Rongchan Zhu}
	\address[R. Zhu]{Department of Mathematics, Beijing Institute of Technology, Beijing 100081, China
	}
	\email{zhurongchan@126.com}

	\begin{abstract}
		The mean-field limit of interacting diffusions without exchangeability, caused by weighted interactions and non-i.i.d. initial values, are investigated. The weights could be signed and  unbounded.  The result  applies to   a large class of singular kernels  including the Biot-Savart law. We demonstrate a  flexible type of mean-field  convergence, in contrast to the typical convergence of $\frac{1}{N}\sum_{i=1}^N\delta_{X_i}$. More specifically,  the sequence of signed empirical measure processes with arbitrary uniform $l^r$-weights, $r>1$, weakly converges to a coupled PDE's, such as the dynamics describing the passive scalar advected by the  2D Navier-Stokes equation.

Our method is based on a tightness/compactness argument and makes use of the systems' uniform Fisher information. The main difficulty  is to determine how to propagate the regularity properties of the limits of empirical measures in the absence of the DeFinetti-Hewitt-Savage theorem for the non-exchangeable case. To this end, a sequence of random measures, which merges weakly with a sequence of weighted empirical measures and has uniform Sobolev regularity,  is constructed through the disintegration of the  joint laws of particles.

	\end{abstract}

	
	\date{\today}
	
	\maketitle
	
	\setcounter{tocdepth}{1}
	
	\tableofcontents

\section{{{}}Introduction}
In this article  we consider non-exchangeable interacting particle systems with singular kernels  in the Euclidean space $\mathbb{R}^d$. Given random initial data $\{X_i(0)\}_{i=1}^N$, the position of each particle $X_i$ is characterized by the following  SDEs\begin{align}
\dif X_i = & \frac{1}{N} \sum_{j \neq i} w_{j}^N { K(t,X_i - X_j)} \ud t  +\sqrt{2} \dif B_t^i, \quad
	  i = 1, \ldots, N,
	\label{eq:IPS}
	\end{align}
 where $K$ denotes the interaction kernel,  $(B^i_\cdot )$ are the independent standard Brownian motions on $\mathbb{R}^d$, $d\geq 2, $ and those  $\{w_j^N\}{\subset \mR}$ are non-identical  deterministic weights that  satisfy  the following assumption for some $r\in (1,\infty]$
\begin{align}
(\mW_r):	\quad \frac{1}{N}\sum_{j=1}^N |w_{j}^N|^r=O(1),\quad \text{for }r<\infty; \quad \max_{1\leq  j\leq N} |w_{j}^N|=O(1), \quad \text{for }r=\infty, \quad \text{as } N\rightarrow \infty,\label{con:wights}
\end{align}
where $O(\cdot)$ means ~``proportional to''.  Here $r>1$  ensures that the system is weakly interacting, indeed   ${|w_j^N|}/{N}\rightarrow 0$. 

One guiding example of the interacting particle system \eqref{eq:IPS} is  the famous stochastic vortex model with general intensities, i.e. the kernel $K$ in the system \eqref{eq:IPS} is the Biot-Savart law defined by
\begin{align}\label{eq:K}
	K= \nabla^{\perp}G=(-\partial_2G,\partial_1G), 
\end{align}
where $G$ is the Green function of the Laplacian on $\mathbb{R}^2$.  Note in particular that
\begin{align*}
	K(x) = \frac{1}{2\pi} \frac{x^\bot}{|x|^2}
\end{align*}
where $x^\bot = (x_1, x_2)^\bot = (-x_2, x_1) \in \mathbb{R}^2$. Now the weights $w_j^N$ denotes the intensity/magnitude of the $j$-th point vortex at position $X_j$. One may expect that now the (weighted) empirical measure defined as 
\begin{equation*}
\mu_N(t) = \frac 1 N \sum_{j=1}^N w_j^N \delta_{X_j(t)}
\end{equation*}
will converge to the  (weak) solutions of the 2D Navier-Stokes equations in the vorticity form 
\begin{equation}\label{2D-NSE_new}
\partial_t v(t, x) + \div ( v(t, x) K * v(t,x) ) = \Delta v(t, x).  
\end{equation} 
Our main results  validate  the above mean-field approximation for 2D Navier-Stokes equation under  very general assumptions on the intensities $w_j^N$: now those $(w_j^N)$ can be of mixed-sign and unbounded. See in particular Theorem \ref{thm_point_vortex}. 

Many-particle systems written in the canonical form \eqref{eq:IPS}  or its  variant are now quite ubiquitous. Such systems are usually formulated by the first-principle agent-based models which are conceptually simple. For instance, in physics those particles $X_i$ can represent ions and electrons in plasma physics \cite{dobrushin1979vlasov}, or molecules in a fluid \cite{jabin2004identification} or even large scale galaxies \cite{Jean} in some cosmological models; in biological sciences, they typically model the collective behavior of animals or micro-organisms (for instance flocking, swarming and chemotaxis and other aggregation phenomena \cite{carrillo2014derivation}); in economics or social sciences particles are usually individual ``agents” or ``players” for instance in opinion dynamics \cite{friedkin1990social} or in the study of mean-field games \cite{lasry2007mean,huang2006large}. Motivation even extends to the analysis of large biological \cite{bossy2015clarification} or artificial \cite{mei2018mean} neural networks in neuroscience or in machine learning.

The classical and more recent investigations on the topic of mean-field approximation  have mainly  focused on the case $w_j^N \equiv 1$ for all  $1 \leq j\leq N$. In this case, under mild assumptions, it is well-known  
(see for instance \cite{mckeanpropagation,braun1977vlasov,dobrushin1979vlasov,osada1986propagation,sznitman1991topics,fournier2014propagation,jabin2018quantitative,serfaty2020mean,jabin2014review,bresch2020mean} )  that  the (usual) empirical measure  $\frac{1}{N}\sum^N_{i=1}\delta_{X_i(t)}$ of  the particle system \eqref{eq:IPS} converges to  the solution $v_t $ to the  nonlinear mean-field PDE \eqref{2D-NSE_new} as $N \to \infty$.  In particular,  mean-field limit on  exchangeable systems (i.e. $w_j^N\equiv1$  and  given the i.i.d. initial data)  is  equivalent to  {\em propagation of chaos}, i.e.  the $k-$marginal distribution  of the particle system  converges to the tensor product of the limit law $g^{\otimes k}$  as $N$ goes to infinity.

Classically, mean-field limit   implies that a  continuum model can be found to approximate the associated particle system  when  $N$ is large.  In this article, we not only establish   mean-field limit for systems with general weights $w^N\assign (w_1^N,...,w_N^N)$ as in the system \eqref{eq:IPS}, as a byproduct, we demonstrate a more flexible mean-field  convergence. Indeed, we can consider the following continuum model given by two coupled PDE's 
\begin{equation}\label{eqt:coupled pde}
	\begin{cases}
			\partial_t g_t =\Delta g_t - \udiv\( g_t \,  K*v_t\) ,  \\
			\\
		\partial_t v_t= \Delta v_t-\udiv \(v_t K*v_t \).
	\end{cases}
\end{equation}
 The continuum  model \eqref{eqt:coupled pde} turns out to be a suitable mean-field system for  the linear statistics of  the interacting diffusions  \eqref{eq:IPS}.
Formally, let $\{\tilde{w}^N\}$ be any other sequence of weights that satisfies the assumption  $(\mathbb{W}_r)$,  then the system \eqref{eqt:coupled pde} is a continuous approximation to \eqref{eq:IPS},  in the sense that as $N \to \infty$,
\begin{align*}
\frac{1}{N}\sum_{i=1}^Nw_i^N\delta_{X_i}\overset{d}{=} v+o(1),\quad \frac{1}{N}\sum_{i=1}^N\tilde{w}_i^N\delta_{X_i}\overset{d}{=}g+o(1),
\end{align*}
where $   \overset{d}{=}$ means that the approximation holds in the sense of  distribution. The novelty of  this approximation is that now the choice of weights  $\tilde{w}^N$ can be quite flexible, including the classical average type, i.e. $\tilde{w}_i^N=1$ for all $i$,  or more general choice based on the relative importance of each particle. When $K$ is the Biot-Savart law, our main results  not only  provides viscous vortex model approximation to the vorticity formulation of the 2D Navier-Stokes equation but now the vorticity is of mixed sign, but also establish a particle approximation to the related passive scalar equation where the flow is given by the Navier-Stokes equation.

\begin{remark}
	  If originally we write that $Z_i = (X_i, w_i^N )$, then the expected mean field limit for the extended space reads that
\begin{equation*}
	\partial_t f (t, x, w) + \udiv_x \Big( f(t, x, w ) \int_{\mathbb{R}^d \times \mathbb{R}}  w' K (x - y) f(t, \ud y, \ud w' ) \Big)= \frac 1 2 \Delta_x f (t, x, w).
	\end{equation*}
The 2nd equation in \eqref{eqt:coupled pde} is related to the above one by
\begin{equation*}
	v(t, x)= \int_{\mathbb{R}}  w f(t, x, \ud w).
\end{equation*}
Then a  natural idea to establish the mean field limit from \eqref{eq:IPS} towards  \eqref{eqt:coupled pde} is to do so  in the extended phase space first then project to the observable ($v$ or $g$)  in a  lower dimensional space as  what have been done for  systems with smooth interaction kernels in \cite{crevat2019mean}  by the classical Dobrushin's estimate \cite{dobrushin1979vlasov}.  See also the examples in \cite{piccoli2014generalized,rotskoff2019global} for particle systems with even evolutionary weights.  We leave the study of  the system \eqref{eq:IPS} with singular $K$ and time-varying $w_j^N= w_j^N(t)$ for future work.

{ It is instructive to contrast our setting with situations where propagation of chaos might be expected.  If the initial random variables $(Z_i(0))$ were i.i.d., the entire system would be exchangeable, and one could pursue a proof of propagation of chaos in this extended space for $Z_i$.  This is precisely the setting in \cite{fournier2014propagation}, where weights are sampled i.i.d. from a common distribution. However, the core feature of our work is that the weights $(w_i)$ are deterministic and non-identical. This renders the agents distinguishable from the outset. Consequently, the joint law of $(Z_1, \dots, Z_N)$ is not symmetric under permutation of the indices. This intrinsic non-exchangeability breaks the standard propagation of chaos framework and constitutes the central challenge we address.}
\end{remark}

\subsection{Main results}
To state our main  result  in a concrete way, we  first give the  definitions  of solutions to the particle system \eqref{eq:IPS} and the mean-field PDEs \eqref{eqt:coupled pde}.

We shall use  the (non-normalized) Boltzmann entropy functional on $\mathcal{P}_{\gamma}(\mathbb{R}^{dN})$, which is the subspace of the probability measure space $\mathcal{P}(\mathbb{R}^{dN})$ under the constraint of  finite $\gamma$-th moment,  $\gamma\in (0,1)$. The entropy functional  is given by
\begin{align*}
	H(f)\assign \int_{\mathbb{R}^{dN}} f\log f\mathd x^N, \quad f\in \mathcal{P}_{\gamma}(\mathbb{R}^{dN})\cap L^1(\mathbb{R}^d),
\end{align*}
with $x^N$ denoting $ (x_1,...,x_N)$, where and in the sequel we also write $f$ as its density for notation's simplicity.  If $f$ has no density, then we set  $H(f)=+\infty$. A well-known fact is that the negative part of $H(f)$ is bounded by a universal constant plus the $\gamma$-th moment of $f$. Thus the entropy functional  is well-defined on $\mathcal{P}_{\gamma}(\mathbb{R}^{dN})$.

We assume  the following conditions on the initial value and the interaction kernel.
\begin{description}
	\item[ $(\mathbb{H})$]Let  $F_0^N$ be the joint distribution of $X^N(0)\assign(X_1(0),..,X_N(0))$. There exists some constant $\gamma>0$ such that
	\begin{align*}
	 H(F_0^N)+	\sum_{i=1}^N\mathbb{E}\<X_i(0)\>^{\gamma}= O(N),
	\end{align*}
where $\<x\>\assign (1+|x|^2)^{\frac{1}{2}}$. 
	\item[$(\mathbb{K}_r)$] Given  $r\in (1,\infty]$  from $(\mathbb{W}_r)$ and $d\geq 2$. The $\mathbb{R}^d$-valued kernel $K$ is of the form   $K=K_1+K_2$,  with $K_1,K_2$ satisfying
	\begin{enumerate}\label{assum:kernel}
		\item $\div K_1,\,  K_1 \in L^{q_1} ([0,T],L^{p_1}(\mathbb{R}^d)) $ with  $\frac{d}{p_1}+\frac{2}{q_1}+\frac{2}{r}<2$, where  the equality can be attached when $q_1, r<\infty$;
		
		\item $K_2\in L^{q_2}([0,T],L^{p_2}(\mathbb{R}^d))$ with $\frac{d}{p_2}+\frac{2}{q_2}+\frac{1}{r}< 1$, where  the equality can be attached when $q_2, r<\infty$.
	\end{enumerate} 
\end{description}

The assumption $(\mathbb{H})$ requires the initial values to satisfy that the entropy and moment estimates are proportional to the number of particles, this is very natural. From $(\mathbb{H})$, we can later obtain that the Fisher information of the particle system is proportional to the number of particles. For kernels that are regular, such as Lipschitz, this assumption is not necessary. However, for singular kernels, we need to assume $(\mathbb{H})$ here to obtain the regularity of the joint distribution, which can be applied to control singular kernels, see Corollary \ref{coro:averagesymme} below. Next, we give typical examples satisfying condition $(\mK_r)$. 

\textbf{Examples}
	\begin{enumerate}
		\item  The Biot-Savart law in dimension 2 as in \eqref{eq:K}.  It is divergence free and  belongs to $L^{p}+L^{\infty}$ with $1<p<2$, so that it satisfies $(\mathbb{K}_r)$ with $r>2$.
		
		\item $K(x)= \frac{x}{|x|^{\alpha}}$ or $-\frac{x}{|x|^{\alpha}}$ with $\alpha\in (1,2)$ and $d\geq 2$.    Then $K\in L^{p}+L^{\infty}$ with $1<p<\frac{d}{\alpha-1}$ and  $(\mathbb{K}_r)$ holds with  $r>\frac{1}{2-\alpha}$.  Since the weights are allowed to  be of  mixed signs, the force  $K$ could be attractive or repulsive. 
	\end{enumerate}


For the particle system \eqref{eq:IPS} on $[0,T], \,  T > 0$, we define the notion of entropy solutions.

\begin{definition}[Entropy Solutions]\label{def:entropy}
Let $X^N=(X_1,...,X_N)$ be a $C([0,T],\mathbb{R}^{dN})$-valued {random variable} satisfying the initial condition $(\mathbb{H})$, and denote the law of $X^N(t)$ by $F_t^N$.

	For the system \eqref{eq:IPS} with the   condition $(\mathbb{K}_r)$, we call  $X^N$  an entropy solution if there exists a universal constant $C>0$ and a stochastic basis $(\Omega, \mathcal{F}, (\mathcal{F}_t)_{t\geq0},\mathbb{P})$ with  a standard $dN$-dimensional Brownian motions $(B_1,...,B_N)$ such that $X^N$ satisfies  the system \eqref{eq:IPS} $\mP$-almost surely and  for $t\in[0,T]$, it holds 
	\begin{align}\label{eq:ent}
		\frac1NH(F_t^N)+	\frac1N\sum_{i=1}^N\mathbb{E} \<X_i(t)\>^{\gamma} +\frac{1}{2 N} \int_0^{{t}}\int_{\mathbb{R}^{dN}} \frac{|\nabla F_t^N|^2}{F_t^N}\mathd x^N \mathd t \leq \frac1N	H(F_0^N)+\frac1N	\sum_{i=1}^N\mathbb{E}\<X_i(0)\>^{\gamma}+ C,
	\end{align}
where and in the sequel $F_t^N$ also denotes its density. 
\end{definition}
Clearly, each entropy solution to \eqref{eq:IPS} is a probabilistically weak solution. The next result gives the existence of entropy solutions.
\begin{proposition}[Proposition \ref{prop:entropy} below]\label{prop}
Under the conditions $(\mathbb{H}) $, $(\mathbb{K}_r)$, and $(\mathbb{W}_r)$  for some $r\in (1,\infty]$,  for each $N\in \mathbb{N}$, there exists an entropy solution $X^N$ to the particle system \eqref{eq:IPS} such that the entropy dissipation inequality \eqref{eq:ent}  holds with some universal constant $C$ that is independent of $N$.
\end{proposition}
The regularity of entropy solutions with a uniform constant, in particular the Fisher information,  would enable us to deal with the singular kernel and  find the mean-field limits.  The entropy solution has  been shown useful for  studying interacting diffusions, and  Proposition \ref{prop} is indeed analogous to  \cite[Proposition 5.1]{fournier2014propagation} and \cite[Proposition 1]{jabin2018quantitative}, but we do not need the divergence free or  bounded-like   conditions on the kernel. 

The well-posedness of the singular interacting system  \eqref{eq:IPS}  with general weights is a fascinating and challenging problem. To the best of the authors' knowledge, the existing results concern specific kernels, such as  \cite{fontbona2007paths}  and \cite{flandoli2011full}  on the Biot-Savart law.  There are a lot of results for identical  weights, for example  \cite{osada1985stochastic,takanobu1985existence,marchioro2012mathematical} on the Biot-Savart law and the recent result \cite{hao2022strong,tomavsevic2023propagation} on $L^q([0,T],L^p(\mathbb{R}^d))$-kernels with $d/p+2/q<1$. However, under the condition ($\mathbb{K}_{\infty}$), even the well-posedness result for associated  SDE: $X_t=X_0+B_t+\int_0^tK(X_s)\mathd s$  remains open,  where the difficulty  comes from the singular kernel $K_1$.  The existence of   probabilistically  weak solutions to the SDE with kernel $K_1$ in the assumption  ($\mathbb{K}_{\infty}$) has been shown in  \cite{zhang2021stochastic}.

We consider the solution in the space $C([0,T], \mathcal{M}(\mathbb{R}^d))$ for the mean-field PDE system \eqref{eqt:coupled pde}. Here  $\mathcal{M}(\mR^d)$ stands for  the space of finite signed measures on $\mR^d$ with the  topology induced by bounded and continuous (test) functions.  We use it as   the state space for the convergence in our main results below.  The solutions to \eqref{eqt:coupled pde} are defined as follows.

\begin{definition} \label{def:pde}We call  $(v,g)\in C([0,T], \mathcal{M}(\mathbb{R}^d))^{\otimes 2}\cap L^{\infty}([0,T],L^1(\mathbb{R}^d))^{\otimes 2}$  a solution to the system \eqref{eqt:coupled pde} if $(v,g)$ satisfies \eqref{eqt:coupled pde}  in the distributional sense and the following estimate holds
	\begin{align}
	 \|v\|_{L^p_q} +  \|g\|_{L^p_q}<\infty ,	\end{align}
for any $p,q$ satisfying 
$$\frac{d}{p}+\frac{2(r-1)}{r}\geq d,\text{ }\frac{d}{p}+\frac{2}{q}\geq d,\text{ } 1\leq p,q<\infty,$$
where $\|\cdot\|_{L^p_q}\assign \|\cdot\|_{L^q([0,T],L^p(\mathbb{R}^d))}$ and $r\in (1,\infty]$. We denote $\frac{r-1}{r}\assign1$ when $r=\infty$.
\end{definition}

\begin{remark}
	The conditions on $p,q,r$ in the above definition ensures that the nonlinear term in the coupled PDEs \eqref{eqt:coupled pde} is well-defined, and also enables  us to deduce uniqueness later. In particular, the uniqueness for the vorticity form of the 2D { Naiver-Stokes} equation follows by applying the results in  \cite{fournier2014propagation, brezis1994remarks,ben1994global}, and see also \cite{gallagher2005uniqueness} for the case with a finite measure  as initial vorticity.

\end{remark}

The first main  result shows  that \eqref{eqt:coupled pde} characterizes the mean-field limits of the interacting system \eqref{eq:IPS}.
\begin{theorem}\label{thm:mfpde}
Let $\{{w}^N\}$ and $\{\tilde{w}^N\}$ be two sequences of weights satisfying the condition $(\mathbb{W}_r)$ with $r\in (1,\infty]$ and suppose that the conditions $(\mathbb{H}) $, $(\mathbb{K}_r)$ hold. Let $X^N$ be an entropy solution to \eqref{eq:IPS}  given by  Proposition \ref{prop}. Assume that there exist $v_0, g_0\in L^1(\mathbb{R}^d)$ such that 
\begin{align}\label{con:initial}
	 \frac{1}{N}\sum_{i = 1}^N w_i^N\delta_{X_i(0)}\rightharpoonup v_0, \quad 	 \frac{1}{N}\sum_{i = 1}^N \tilde{w}_i^N\delta_{X_i(0)}\rightharpoonup g_0
\end{align}
in 
$\mathcal{M}(\mathbb{R}^d)$ almost surely, where $\rightharpoonup$ means the weak convergence in $\cM(\mR^d)$.

It holds that the corresponding family of laws for   the weighted empirical measures $(\mu_N, \tilde{\mu}_N)$ defined by
\begin{equation}\label{def:Empirical}
	\mu_N(t)\assign \frac{1}{N}\sum_{i=1}^Nw_i^N \delta_{X_i(t)}, \quad 	\tilde{\mu}_N(t)\assign \frac{1}{N}\sum_{i=1}^N\tilde{w}_i^N \delta_{X_i(t)},
\end{equation}
is tight in  $C([0,T], \mathcal{M}(\mathbb{R}^d))^{\otimes 2}$ and every accumulation point 
is a solution to  \eqref{eqt:coupled pde} with initial value $(v_0,g_0)$ in the sense of Definition \ref{def:pde}. 
\end{theorem}

{Here, the term ~``accumulation point" refers to the limit of a convergent subsequence. Indeed, due to the tightness of the sequences $\mu_N$ and $\tilde{\mu}_N$, we can apply Prokhorov’s theorem followed by the Skorokhod representation theorem to extract convergent subsequences. Moreover, one can verify that every such limiting point satisfies equation \eqref{eqt:coupled pde}.  Note that in the absence of a uniqueness result for \eqref{eqt:coupled pde}, the limiting point  may be a random solution.}

\begin{remark*}
	Theorem \ref{thm:mfpde} derives  \eqref{eqt:coupled pde} from the particle system \ref{eq:IPS}. On the other hand, given a solution to  \eqref{eqt:coupled pde}  with the initial data  $v_0, g_0\in L^1(\mathbb{R}^d)$, one could also construct  a suitable sequence of   weight $\{w^N\}$ satisfying $(\mW_r)$ with $r=\infty$ and initial variables $X^N(0)$  such that \eqref{con:initial} holds, therefore 	Theorem \ref{thm:mfpde} can be applied when the uniqueness of mean-field limits holds. For instance, given $v_0\in L^1(\mathbb{R}^d)$, set $v_0^+=v_0\vee 0$ and $v_0^-\assign v_0^+-v_0$. Let $w_i^N= 2\|v_0^+\|_{L^1}$ for $i=1,...,\frac{N}{2}$ and $w_i^N= -2\|v_0^-\|_{L^1}$ for $i=\frac{N}{2}+1,...,N$; let $(X_1(0),...,X_N(0))$ be a class of independent random variables such that $X_i(0)$ is distributed as ${v_0^+}/{\|v_0^+\|_{L^1}}$ for $i=1,..., \frac{N}{2}$ and $X_i(0)$ is distributed as ${v_0^-}/{\|v_0^-\|_{L^1}}$ for $i= \frac{N}{2}+1,...,N$. Then \eqref{con:initial} can be easily checked by the law of large numbers.
\end{remark*}

Theorem \ref{thm:mfpde}  gives the mean-field limits of interacting diffusions when the kernel is  singular and the system is non-exchangeable  at the same time.  Our results extend classical mean-field limits for singular interacting diffusions  to  non-exchangeable cases, additionally, the weights for the interaction can be  unbounded, i.e., $r<\infty$. The system  \eqref{eq:IPS} with the Biot-Savart law and bounded weights has been studied in  \cite{fournier2014propagation,wynter2021quantitative}.  However, the result in \cite{wynter2021quantitative}  only applies  to the so-called two pieces interaction, which is basically  $w_i^N=a_1>0$ when $i$ is odd and $w_i^N=a_2<0$ when $i$ is even. Under that  restrictive condition, the pairs of particles $(X_{2i-1},X_{2i})$ are exchangeable. {The work of \cite{fournier2014propagation} is closer to our setting, but operates under a critically different assumption. In their framework, the weights $w_i^N$ are random variables drawn i.i.d., which renders the extended variables $Z_i = (X_i, w_i^N)$ exchangeable. This initial symmetry is the cornerstone of their proof, which establishes propagation of chaos for the extended system. In contrast, our work addresses the case of deterministic and heterogeneous weights, where the joint law of the system is not symmetric, making the particles fundamentally non-exchangeable.} 
Additionally, we note that  the convergence result in \cite{wynter2021quantitative} is quantitative, which adapts the relative entropy method as in \cite{jabin2018quantitative}, 
while the results in \cite{fournier2014propagation} and this study, which use tightness argument, are qualitative.

As particles/agents in applications are not always identical, more natural assumptions on weights are required. Let us now look at a specific  example.   Let $K(t,X_i-X_j)$ be  the interaction in a $N$-particles system described by the system of SDEs \eqref{eq:IPS}. The interaction could be viewed as a function of how $X_i$ is influenced by $X_j$, with  $w_j^N$ representing the intensity associated with $X_j$.    Let the weights be $w^{5N }=(w_1^N,...,w_N^N,0,...,0) $ with $(w_1^N,...,w_N^N) $ satisfying
\begin{align*}
	|w_i^N|=O (N^{\frac{1}{3}}),  \quad \forall 1\leq i\leq N^{\frac{1}{3}}; \quad \quad |w_i^N|=O(1), \quad N^{\frac{1}{3}}< i\leq N.
\end{align*}
Now the model is 
\begin{align*}
	\begin{cases}
		\dif Y_i = & \frac{1}{5N} \sum_{j \neq i} w_{j}^N K(Y_i - Y_j) \ud t  +\sqrt{2} \dif B_i, \quad
		i = 1, \ldots, N,	  \\
		\\
		\dif Z_m = & \frac{1}{5N} \sum_{j =1}^N w_{j}^N K(Z_m- Y_j) \ud t  +\sqrt{2} \dif B_m, \quad
		m= 1, \ldots, 4N.
	\end{cases}.
\end{align*}
Here $(Y^N,Z^{4N})$ plays the role of $X^N$ in \eqref{eq:IPS} and $(\mathbb{W}_r)$ holds for $r=2$. Notice that every particle  only interacts with the particles of the type $Y$, in other words, only the particles $(Y_i)$ contribute to the dynamics of the system. Furthermore, there are still  differences among the particles of the type $Y$.  The  particles $(Y_i)$  belongs to  two  groups,  the majority of the number $N-N^{\frac{1}{3}}$ and the minority of the number $N^{\frac{1}{3}}$.  Each particle in the   majority  contributes to the system  in a normal way that $|w_j^N|=O(1)$.
In contrast, those particles from the   minority   make significant contributions, at the scale of $N^{\frac{1}{3}}$.



 {As mentioned earlier, our result on the convergence of the pair $(\mu_N, \tilde{\mu}_N)$ is much more general than the convergence of the standard empirical measure, as it implies convergence for any sequence of weights $(\tilde{w}_j^N)$ under the assumption $(\mathbb{W}_r)$. This notion can be made precise by considering the full empirical measure on the extended space,
 		$$
 		\nu^N(t, \mathd  x, \mathd \xi) := \frac{1}{N}\sum_{i=1}^N \delta_{X_i(t)}(\mathd x) 1_{[\frac{i-1}{N},\frac{i}{N})}(\xi )\mathd \xi,
 		$$
 		which encodes the complete information of all particle positions. The "linear statistical information" of the system $(X_1, \dots, X_N)$ corresponds to testing this measure against functions. Specifically, for any $\varphi \in C_b(\mathbb{R}^d)$ and any continuous function of the particle identity $\phi \in C([0,1])$, we have:
 		$$
 		\langle \nu^N(t), \varphi(x)\phi(\xi) \rangle = \frac{1}{N} \sum_{i=1}^N \left( N \int_{\frac{i-1}{N}}^{\frac{i}{N}} \phi(\xi) \mathd \xi \right) \varphi(X_i(t)).
 		$$
 		The right-hand side is precisely the weighted empirical measure with weights $w_i^\phi := N \int_{\frac{i-1}{N}}^{\frac{i}{N}} \phi(\xi) \mathd \xi$.  Our main result (Theorem \ref{thm:mfpde}) establishes the convergence of these observables for \textit{any} choice of $\phi$, provided the initial data converges. Therefore, our result establishes the convergence of the full empirical measure $\nu^N$ when tested against any function of the product form $\varphi(x)\phi(\xi)$.

The closure of the limiting PDE system \eqref{eqt:coupled pde} is a special feature of the weight structure $w_{ij}=w_j$. For general weights, the limit of any single such observable would not be described by a closed equation but would depend on the limit of $\nu^N$ itself.}

In particular, when $K$ is the Biot-Savart law, let $ u_t=K*v_t$ denote the  velocity of the fluid with $v$ representing the vorticity, then Theorem \ref{thm:mfpde} gives a mean-field approximation to the dynamics of a passive tracer undergoing advection-diffusion in the fluid described by the Navier-Stokes  equation,
\begin{align}\label{eqt:scalar}
	\begin{cases}
		\partial_t u_t= \Delta u_t-u_t\cdot \nabla u_t+\nabla p,	\quad \div u=0,  \\
		\\
	\partial_t g_t =\Delta g_t - u_t \cdot \nabla g_t,
	\end{cases}
\end{align}
with $p$ being the associated pressure. For more information on the physics behind passive scalars, we refer to \cite{falkovich2001particles,shraiman2000scalar,warhaft2000passive} and  references therein, and for mathematical studies, see  e.g.  \cite{alberti2019exponential,bedrossian2022almost,bedrossian2022batchelor,seis2013maximal,zelati2020relation} for instance.

Particle systems  on graphs  have attracted  increasing attention in recent years, including  interacting  diffusions  with  more general $w_{ij}^N$ replacing $w_{j}^N$.  Our work can be thought of as a special case (the weight $w_{ij}^N$ depends only on $j$) of  interacting  diffusions on graphs  with weights  in $l^r$ space.  The mean-field convergence for interacting  diffusions on dense  graphs  has been studied by  \cite{delattre2016note,kaliuzhnyi2018mean,kuehn2022vlasov,bayraktar2020graphon,bet2020weakly,jabin2021mean,lucon2020quenched,oliveira2019interacting,coppini2023central,luccon2014mean,ayi2024mean,jabin2024dense}  in different settings.   However,  to the best of our knowledge their results do not cover  singular kernels.      We will do detailed  comparisions on the  methods of \cite{bayraktar2020graphon, jabin2021mean} in Section \ref{subsec:method} below.   Our work takes a different approach in this direction, although we do not derive the convergence rate.  Motivated by the general form of the stochastic 2D vortex model,  we simply consider   the weights given as $(w_j^N)$ in this work, where $w_j^N$ means the intensity of the  vortex localized in $X_i$. { A core technical contribution of our work is the development of analytical tools to obtain the crucial regularity estimates required to handle singular kernels without leveraging exchangeability. Classical proofs for singular interactions often rely heavily on the symmetry of the system (e.g. \cite{fournier2014propagation}). Our approach bypasses this requirement, which is essential for the non-symmetric systems we consider. This robustness is precisely what makes our method extensible. Indeed, while the hypothesis $w_{ij}^N = w_j^N$ allows for a closed-form limiting system, our analysis would extend in an analogous manner to interacting diffusions on general $l^r$-graphons ($r > 1$), see the second named author's thesis \cite{zhao2023derivation}.}
 Beyond the setting of the  mean-field convergence  for systems on dense graphs,    the  asymptotic behavior of  interacting diffusions  on sparse graphs, which involves strong  interactions,  is another related active topic. The local weak convergence  for the sparse case
has been  obtained  in   \cite{oliveira2020interacting,lacker2019local} recently.

With additional constraints on the kernel and weights, we can demonstrate the uniqueness of the limiting points of converging subsequences and thus obtain the convergence of the entire sequence.

\begin{theorem}\label{thm:uniquepde}
	If  the kernel $K$  belongs to either of the following two cases,
	\begin{enumerate}
		\item $K$ is the 2D Biot-Savart law and  $r\in [3,\infty]$ (since the definition of  entropy solutions depends on $r$).
		\item 	Given  $r\in (1,\infty]$, and $K $ belongs to $    L^{q_2}([0,T],L^{p_2}(\mathbb{R}^d))$ with
		\begin{align}
			\frac{d}{p_2}+ \frac{2}{q_2}+\frac{1}{r}\leq 1,\quad  \frac{d}{p_2}+ \frac{2}{q_2}<1.
		\end{align}
	\end{enumerate}
Then there exists a unique solution $(v,g)$ to \eqref{eqt:coupled pde} for  each given initial value from $L^1(\mathbb{R}^d)^{\otimes 2}$.
\end{theorem}

{ We are now in position to show that the convergence of all subsequential limits to a unique solution implies the convergence of the entire sequence. 
	We emphasize that this argument relies on the topological properties of our space. 
	The space of finite signed measures $\mathcal{M}(\mathbb{R}^d)$ endowed with the weak topology has metrizable compact subsets.  The tightness of the sequence of laws (established in Theorem \ref{thm:mfpde}) ensures that the sequence effectively lives on such a compact, and therefore metrizable, subset. In this setting, the uniqueness of subsequential limits is sufficient to conclude that the full sequence converges.}

\begin{corollary}\label{coro:scalar} Suppose that the condition $(\mathbb{H}) $ holds.
	Given two sequences of weights $\{{w}^N\}$ and $\{\tilde{w}^N\}$ satisfying the condition $(\mathbb{W}_r)$ with $r=3$.  Let $X^N$ be an entropy solution to  the stochastic vortex model  (i.e. Eq. \eqref{eq:IPS} with  $K$  the Biot-Savart law)  given by  Proposition \ref{prop}. Assume that there exist $v_0, g_0\in L^1(\mathbb{R}^d)$ such that 
	\begin{align}
		\frac{1}{N}\sum_{i = 1}^N w_i^N\delta_{X_i(0)}\rightharpoonup v_0, \quad 	 \frac{1}{N}\sum_{i = 1}^N \tilde{w}_i^N\delta_{X_i(0)}\rightharpoonup g_0\nonumber
	\end{align}
	in 
	$\mathcal{M}(\mathbb{R}^2)$ almost surely.
	 Then $(\mu_N, \tilde{\mu}_N)$ defined in Theorem \ref{thm:mfpde} converges in law to $(v,g)$ in  $C([0,T],$ $\mathcal{M}(\mathbb{R}^2))^{\otimes 2}$. Here $(v,g)$ uniquely solves \eqref{eqt:coupled pde} in the sense of Definition \ref{def:pde}. In particular, $(\nabla^{\perp}(-\Delta)^{-1} v,g)$ solves the system  \eqref{eqt:scalar} of the passive scalar advected by the 2D Navier-Stokes equation.
\end{corollary}

Finally, simply by choosing the same weight sequences $\tilde w_j^N = w_j^N$ as the intensities of the $j-$th point vortex, one arrives at the following theorem. 
\begin{theorem}[Mean field limit for stochastic vortex model with general intensities] \label{thm_point_vortex} Suppose that the condition $(\mathbb{H}) $ holds.
	Given a sequence of intensities $w^N = (w_j^N)_{1 \leq j \leq N}$ that satisfies  the condition $(\mathbb{W}_r)$ with $r=3$.  Let $X^N$ be an entropy solution to  the stochastic vortex model  \eqref{eq:IPS} with  $K$  the Biot-Savart law.  Assume that the initial data for the 2D Navier-Stokes equation \eqref{2D-NSE_new}  $v_0\in L^1(\mathbb{R}^d)$ and  
	\begin{align}
	\frac{1}{N}\sum_{i = 1}^N w_i^N\delta_{X_i(0)}\rightharpoonup v_0, \nonumber
	\end{align}
	in 
	$\mathcal{M}(\mathbb{R}^2)$ almost surely.
	Then the empirical measure $\mu_N = \frac 1 N \sum_{i=1}^N w_i^N \delta_{X_i}$ converges in law to $v$ in  $C([0,T],$ $\mathcal{M}(\mathbb{R}^2))$, where $v$ is the  unique solution to the second equation of  \eqref{eqt:coupled pde} in the sense of Definition \ref{def:pde}. 
\end{theorem}

Mean field limit and propagation of chaos for the 1st order system given in the  form  \eqref{eq:IPS} with $w_j^N=1$ have been extensively studied over the last decade. The basic idea of deriving some effective PDE describing the large scale behaviour of interacting particle systems dates back to Maxwell and Boltzmann. But in our setting, the very first mathematical investigation can be traced back to McKean in \cite{mckeanpropagation}. See also the classical mean field limit from Newton dynamics towards Vlasov Kinetic PDEs in \cite{dobrushin1979vlasov,braun1977vlasov,jabin2015particles,lazarovici2016vlasov} and  the review \cite{jabin2014review}. Recently much progress has been made in the mean field limit for systems as \eqref{eq:IPS} with $w_j^N=1$ and singular interaction kernels, including those results focusing on the vortex model \cite{osada1986propagation,fournier2014propagation}  and very recently quantitative convergence results  on  general singular kernels  for example as in \cite{jabin2018quantitative, bresch2020mean,guillin2024uniform,guillin2022systems,lacker2021hierarchies,serfaty2020mean,duerinckx2016mean,rosenzweig2020mean,nguyen2021mean}.  See also the references therein for more complete development on  mean field limit.

In particular, as we studied in  Theorem \ref{thm_point_vortex},  the point vortex approximation towards the 2D Navier-Stokes equation has aroused much interest since the 1980s. Osada \cite{osada1986propagation} firstly  obtained a  propagation of chaos result  with bounded initial distribution and large viscosity.  More recently, Fournier, Hauray, and Mischler \cite{fournier2014propagation} obtained entropic propagation of chaos by the compactness argument, and  their result applies to all viscosity  and all initial distributions with finite $\gamma$-th moment ($\gamma>0$) and finite Boltzmann entropy.  Jabin and Wang have established a quantitative estimate of the propagation of chaos in \cite{jabin2018quantitative} by evolving  the relative entropy between  the joint distribution of $X^N$ and the tensorized law at the limit. Very recently, the authors proved the Gaussian fluctuations in \cite{wang2023gaussian} by a tightness argument. We also mention the recent  large deviation result  obtained by Chen and Ge \cite{chen2022sample}.  However, we are not aware of any mean-field approximations to the system of the  passive scalar \eqref{eqt:scalar}.
Given that $K\in L^p_q$ with $d/p+2/q< 1$, the corresponding case of homogeneous weights ($w_j^N \equiv 1$), { which describes a classical exchangeable particle system,}  has been extensively studied, particularly  the mean-field convergence obtained by various approaches (e.g., \cite{hao2022strong,tomavsevic2023propagation}); we specifically  refer to  \cite{hoeksema2024large} for the study of   large deviations.

\subsection{Difficulties and Methodology}\label{subsec:method}
The property of exchangeability is crucial  among scaling limits of   interacting diffusions, particularly those with  singular interactions. The difficulty is to compare interacting particle systems with the mean-field limits in the absence of exchangeability.   As mentioned before, the results in \cite{bayraktar2020graphon,jabin2021mean} can be applied to the weights given by the  general form $w_{i j}^N$,   and  the symmetry of $w_{i j}^N$ is not required in \cite{jabin2021mean}. In both papers the coupling method and graph theory are used.     The basic idea in the  proof of \cite{bayraktar2020graphon} is  comparing the SDEs of the particles and the SDEs of the limiting system, using  the convergence of graphons (referring to \cite{lovasz2012large,borgs2019L}).  In contrast, the coupling method was used in \cite{jabin2021mean} to obtain  a coupled PDEs of the McKean-Vlasov type by propagation of independence  first. Then a class of new observables was constructed through the combinations of  weights and laws  of  independent particles via  a family of labeled trees. They then transformed the problem into the Vlasov/mean-field hierarchy  likewise. Instead of directly using the convergence of graphons, a similar version of Szemer\'edi's regularity lemma  was established in \cite[Lemma 4.7]{jabin2021mean} for non-exchangeable systems.
However, the coupling method   leads to the bounded and Lipschitz continuity restriction on the interactions $K$. { As for other approaches like the relative entropy method \cite{jabin2018quantitative} or the modulated energy method \cite{serfaty2020mean}, a direct application is hindered by the fact that our key observables $\mu^N$ are signed measures. While one might consider applying these methods to the empirical measure on the extended space, this approach has so far been successful only in the exchangeable regime. For instance, a very recent work \cite{feng2024quantitative} extends the relative entropy method to the 2D vortex model with general circulations, but still relies on the initial symmetry of the extended system. How to adapt these powerful techniques to the non-exchangeable setting remains a question. Our work provides an answer for such systems by developing a direct compactness argument.}

Our idea of the proof is to apply the compactness/tightness argument, consisting of the classical three steps: the tightness, characterizing the limits, and the uniqueness of the limit equation. For particle systems with singular interactions, the regularity of the  joint law is crucial to compensate the singularity (see Corollary \ref{coro:averagesymme} below).  The article by Fournier, Hauray, and Mischler \cite{fournier2014propagation} is probably the one closest to our article regarding tackling the singularity, where they also heavily exploited the Fisher information of the joint laws of particles. 
 Many features of the Fisher information will also be used in our proof, for instance super-additivity, the chain rule, and notably the Sobolev regularity estimates from it. { Unfortunately,  the analysis in \cite{fournier2014propagation} fundamentally relies on the assumption that the initial law of the system is chaotic. In contrast, our framework does not require chaoticity. We instead impose a weaker condition based on a uniform initial bound (Hypothesis ($\mathbb{H}$)), which allows for non-symmetric joint laws and correlated initial particle positions.} It is not apparent how to apply the compactness argument, since the non-exchangeability and the singularity cause two difficulties. The first one is to derive uniform estimates (about the Fisher information in our case)  for the non-exchangeable system \eqref{eq:IPS}. For instance, when $K$ is the Biot-Savart law,
the following frequently used inequality requires exchangeability,
\begin{align*}
	 \mathbb{E} \left| K(X_i-X_j)\right| \lesssim 1+I(\mathcal{L}(X_i,X_j))\lesssim1+\frac{2}{N}I(\mathcal{L}(X^N)),
\end{align*}
where $I(\mu):\mathcal{P}(\mathbb{R}^{dk})\rightarrow [0,+\infty]$, $ k\in \mathbb{N}$, denotes the Fisher information functional to be defined in Section \ref{sec:notations}. This indicates that the  interaction between any two particles   can be controlled by $\frac{1}{N}$ of the total Fisher information of the system. However, when the particles are {\em not} indistinguishable  from each other, the joint law is no longer symmetric, and there might be a pair of particles such that $I(\mathcal{L}(X_i,X_j))> \frac{2}{N}I(\mathcal{L}(X^N))$.
To overcome this difficulty, we build a technical lemma (Lemma \ref{lemma:sym}) concerning the average of the interactions,  which allows us to derive uniform Fisher information for non-exchangeable systems and  estimate  singular interactions. Investigating averaging statistics is  a key step to study non-exchangeable systems. Similarly,   observables with  averaged information of  particles also play a crucial role in  \cite{jabin2021mean}. Note that the presence of noise is crucial in our analysis since we effectively use the control given by the Fisher information. In the deterministic setting, for instance in the vortex approximation towards the 2D incompressible Euler equation, since now those $w_j^N$ in general can be distinct to each other, the symmetrization trick used in \cite{jabin2018quantitative} does not work anymore.

The other major difficulty is to show the regularity of the limiting points of the empirical measures $\{\mu_N, N \in \mathbb{N}\}$. This is closely related to the exchangeability.  With compactness argument, one usually find $\mu_N$ converges to $\mu$ in the distributional sense. To make the mean-field  equation with singular coefficients well-defined and show the convergence of the nonlinear  interaction term, one must propagate the regularities.  Clearly, the empirical measure $\mu_N$ enjoys no regularity at all. Again, we find  this  difficulty is not problematic for  the symmetric case since there are  well-established tools based on  the famous DeFinetti--Hewitt--Savage theorem \cite{de1937prevision,hewitt1955symmetric} to propagate the Fisher information, c.f. \cite{hauray2014kac} and \cite{fournier2014propagation}.  For the non-exchangeable case, we introduce a sequence of  random measures (see $g^N$ defined in \eqref{defi:gnmarginal}) constructed via disintegration of the joint laws  $\{F^N,N\in \mathbb{N}\}$. This sequence of random measures  would merge with the sequence of empirical measures as $N$ goes to infinity (see Definition \ref{def:4.1} below), thus playing a similar role as 1-marginal distribution  in the symmetric case. An analogous construction can be found in the proof of  Laplace principle via weak convergence method, see e.g. \cite[Section 2.5]{dupuis2011ldp}. The difference is that   the proof in \cite{dupuis2011ldp} studies the relative entropy while  we focus on the Fisher information functional of probability measures. In the end, uniform Sobolev regularity estimates for the random measures are obtained in Section \ref{sec:randommeasure}.  Consequently, we obtain the required regularity of the limiting points.

The remainder of the compactness argument is standard, except that we shall work on the space of finite signed measures instead of the  space of probability measures. It is worth mentioning here that the proof of Theorem \ref{thm:uniquepde} for the Biot-Savart law  relies on the uniqueness result  for the 2D Navier-Stokes equation  in \cite{fournier2014propagation}, which is based on \cite{ben1994global,brezis1994remarks}. {  This direct path to a well-understood macroscopic equation highlights the efficiency of our approach. While an alternative route via the extended phase space as in  \cite{fournier2014propagation} could be considered, its main advantage, the proof of propagation of chaos, is unattainable in our non-symmetric setting, making such a detour less motivated.}

\subsection{Organization of the paper}
This paper is organized as follows. We shall state the notations and auxiliary estimates related to  the Fisher information  in Section \ref{sec:notations}. Section \ref{sec:uniform}  is devoted to obtaining  the main estimate in this article,  which  gives a uniform control on the Fisher information of the joint laws of $N$-particles. The proof is   based on an averaging estimate for the Fisher information when the probability measure is asymmetric. In Section \ref{sec:randommeasure}, we study a sequence of random measures, which turns out to be close to the sequence of weighted empirical measures and enjoys certain Sobolev regularity estimates uniformly.  Lastly, we finish the proofs  of Theorem \ref{thm:mfpde} and Theorem \ref{thm:uniquepde} by the compactness argument in Section  \ref{sec:mf}.
\subsection*{Acknowledgments} We would like to thank Pierre-Emmanuel Jabin for helpful suggestions.  We are also deeply grateful to the anonymous referee for a remarkably detailed and insightful report, which led to a significant improvement of our manuscript.
Z.W. is supported by the National Key R\&D Program of China, Project Number 2021YFA1002800, NSFC grant No.12171009, Young Elite Scientist Sponsorship Program by China Association for Science and Technology (CAST) No. YESS20200028 and the start-up fund from Peking University.
R.Z. is grateful to the financial supports of the National Key R\&D Program of China, Project Number 2022YFA1006300, and the NSFC (No. 12426205, 12271030).
The financial support by the  Deutsche Forschungsgemeinschaft (DFG, German Research Foundation) – Project-ID 317210226--SFB 1283 are greatly acknowledged.
\

\section{Preliminaries}\label{sec:notations}

\subsection{Notations}
Throughout the paper, we use the notation $a\lesssim b$ if there exists a universal constant $C > 0$ such that $a\leq Cb$. During the computations, the universal constant may change from line to line, we will point out the dependence on the parameters when it is  necessary.   Recall that we have used the notations: $X^N\assign(X_1,...X_N)$, $x^N\assign(x_1,...x_N)$, $w^N\assign (w_1^N,...,w_N^N)$, and  $\<x\>\assign (1+|x|^2)^{\frac{1}{2}}$. The $\gamma$-th moment, $\gamma>0$,  of a positive  measure $\mu$ on $\mathbb{R}^d$  is represented by  $\int_{\mathbb{R}^d}\<x\>^{\gamma}\mu(\mathd x )$. As usual, $C_0(\mathbb{R}^d)$ stands for the space of  continuous functions vanishing at infinity, and $C_0^{\infty}(\mathbb{R}^d)$ stands for the space of  smooth functions vanishing at infinity. Let $\mathcal{S}(\mathbb{R}^d)$ be the Schwartz space. We also use $C_b^k(\mR^d)$ to denote the  space of bounded continuous functions with bounded $k$-th derivative. We use  $\|\cdot\|_{L^p_q}$ to denote the $ L^q([0,T],L^p(\mathbb{R}^d))$-norm. For the notation's simplicity   we shall not distinguish the space and the norm for the vector valued functions and the scalar valued functions.  For $r\in [1,\infty)$, we define
\begin{align*}
	\|w^N\|_{l^r}\assign\bigg( \frac{1}{N}\sum_{i=1}^N|w_j^N|^{r}\bigg) ^{\frac{1}{r}},
\end{align*}
and   $\|w^N\|_{l^{\infty}}\assign \max_{1\leq  j\leq N} |w_j^N|$. Obviously, $\|w^N\|_{l^{r_1}}\leq \|w^N\|_{l^{r_2}}$ when $r_1\leq r_2$.

We use $\cP(\mR^d)$ to denote the probability space on $\mR^d$ and for a given  Polish space $\mathcal{X}$ we  use $\sB(\cX)$ to denote the Borel $\sigma$-algebra on $\cX$ and use $\mathcal{M}(\mathcal{X})$ to denote the space of finite signed measures   endowed with   the weak topology induced by bounded, continuous functions on $\cX$, i.e. the convergence in $\cM(\cX)$ is equivalent to the convergence testing with  bounded continuous functions on $\cX$.  Given $\mu\in \mathcal{M}(\mathcal{X})$, its absolute value is denoted by $|\mu|$, i.e. $|\mu|\assign\mu^{+}+\mu^{-}$. We use $\|\cdot\|_{TV}$ to denote the total  variation norm of elements in $\mathcal{M}(\mathcal{X})$.    The notation $L^p_w(\mathcal{X})$, $p\geq 1$,  denotes  the $L^{p}(\mathcal{X})$ space endowed with the weak topology induced by its dual space.

\subsection{Entropy and Fisher information functionals }
In this section, we define the  Fisher information  functional for $N$-particle distribution functions, and collect some related auxiliary estimates.

For $F\in \mathcal{P}(\mathbb{R}^{dN})$, the (non-normalized) Fisher information functional is defined as follows
\begin{align}\label{def:en}
  I(F) :=\int_{\mathbb{R}^{dN}}\frac{|\nabla F(x^N)|^2}{F(x^N)}\mathd x^N.\nonumber
\end{align}
When $F$ has no density, we set $I(F)=+\infty$.

The following lemma (a modification of \cite[Lemma 3.7]{hauray2014kac} ) shows the Fisher information of probability measures is
super-additive.  The general study of this topic can be found in \cite{hauray2014kac}. See also  \cite[Theorem 3]{carlen1991superadditivity}  for an analytic proof.
\begin{lemma} \label{lem:summarginal}
 Let $X^N \assign (X_1, \ldots X_N)$ be a random variable on
  $\mathbb{R}^{dN}$ with joint law $F^N$, and denote by $F_i$    the law of $X_i$, or more precisely
   \begin{align}
  	F_i (\dif x_i) = & \int_{\mathbb{R}^{d (N - 1)}} F^N (\dif x_1 \cdots \mathd x_{i-1} \mathd x_{i+1} \cdots
  	\dif x_N) . \nonumber
  \end{align}
  Then it holds that
  \begin{align}
    \sum_{i=1}^N  I (F_i) \leqslant & I (F^N) \nonumber
  \end{align}
where  $I(F_i)$ is the Fisher information for distributions in $\mathcal {P}(\mR^d)$, while $I(F^N)$ is the one for the joint law $F^N \in \mathcal{P}(\mR^{dN})$.
\end{lemma}

\begin{proof}
  One could assume $I (F^N)$ is finite, and use the variational formulation of
  Fisher information (see \cite[Lemma 3.5]{hauray2014kac}) to obtain 
  \begin{align}
    I(F^N) = & \sup_{\varphi \in C^1_b
    (\mathbb{R}^{dN};\mathbb{R}^{dN} )} \left\langle F^N, - \frac{| \varphi
    |^2}{4} - \tmop{div} \varphi \right\rangle \nonumber\\
    \geqslant & \sup_{\varphi_i \in C^1_b (\mathbb{R}^d;\mathbb{R}^d ), \, 1 \leqslant i
    \leqslant N} \left\langle F^N, - \sum_{i=1}^N  \left( \frac{| \varphi_i |^2}{4} +
    \tmop{div}_i \varphi_i \right) \right\rangle \nonumber\\
    = & \sum_{i=1}^N  \sup_{\varphi_i \in C^1_b (\mathbb{R}^d; \mathbb{R}^d)} \left\langle F_i, -
    \left( \frac{| \varphi_i |^2}{4} + \tmop{div}_i \varphi_i \right)
    \right\rangle \nonumber\\
    = & \sum_{i=1}^N  I (F_i) , \nonumber
  \end{align}
  where $\varphi_i$ depends only on the $i$-th variable.  This completes the proof.
\end{proof}

One can control $L^p$ norms and $W^{1,p}$ norms of a probability density function by its Fisher information. More precisely 
\begin{lemma}\label{lemma:sobo}For  $d \geq 3$, a probability measure  $F$  on $\mathbb{R}^d$ with finite Fisher information, one has
	\begin{enumerate}
		\item For all $p\in [1,\frac{d}{d-2}]$, it holds that 	$\|F\|_{L^p(\mathbb{R}^d)}\leq C_{p,d}I(F)^{\frac{d}{2}(1-\frac{1}{p})}$.
		\item For all $q\in [1,\frac{d}{d-1}]$, it holds that $\|\nabla F\|_{L^q(\mathbb{R}^d)}\leq C_{q,d}I(F)^{\frac{d+1}{2}-\frac{d}{2q}}$.
	\end{enumerate}
Note that if $d=2$, then the 1st control holds for all $p \in [1, + \infty )$, while the 2nd estimate holds for $q \in [1, 2)$.

\end{lemma}
\begin{proof} These estimates are quite standard.  We refer the interested  readers to  Lemma 3.2 in \cite{fournier2014propagation}  for the 2-dimensional case and also Lemma 2.4 in  \cite{li2019mean} for the general case $d \geq 3$.  The proof is essentially based on the interpolation inequality, Sobolev inequality, and also the fact that $\| F\|_{L^1} =1$ since it is a probability density. 	
\end{proof}

\begin{lemma}\label{lemma:lpqf} For $d\geq2$,
consider an $\mathbb{R}^d-$valued function  $\tilde{K}\in L^q([0,T],L^p(\mathbb{R}^d))$ with
	\begin{align*}
		\frac{d}{p}+\frac2q+\frac{2}{r}\leq 2,  \quad \frac{d}{p}+\frac{2}{r}<2,\quad r\in (1,\infty],
	\end{align*} and  probability measures $F(t, \cdot )$ on $\mathbb{R}^d$ with finite Fisher information for a.e. $t \in [0, T]$.  Then for any $\eps>0$, we have
	\begin{align}
		\int_0^T\int_{\mathbb{R}^{d}}|\tilde{K}(t, x)|^{\frac{r}{r-1}} F(t, x)\mathd x \dif t\leq \|\tilde{K}\|_{L^p_q}^{\frac{r}{r-1}} \bigg(C_{\eps,p,q,r,d, T} +\eps \int_0^T I(F(t, \cdot))\mathd t\bigg).
	\end{align}
\end{lemma}
\begin{proof}When $p=+\infty$, the result is trivial. So we prove only   for the case when   $p<\infty$, and we then have  $1/q+1/r< 1$. Notice that when $d \geq 2$,   the condition $d/p+2/r<2$ implies that  $p>r/(r-1)$.
	
Repeatedly applying H\"older's inequality  gives
	\begin{align}
	\int_0^T\int_{\mathbb{R}^{d}}|\tilde{K}(t, x)|^{\frac{r}{r-1}} F(t, x)\mathd x \dif t \leq &	\|\tilde{K}\|_{L^p_q}^{\frac{r}{r-1}}
	\(\int_0^T\(\|F(t, \cdot )\|_{L^{\frac{p}{p-\frac{r}{r-1}}}}\)^{\frac{q}{q-\frac{r}{r-1}}} \mathd t \)^{\frac{q-\frac{r}{r-1}}{q}}\nonumber
	\\\leq &C_{p,r,d,T}	\|\tilde{K}\|_{L^p_q}^{\frac{r}{r-1}}
	\(\int_0^T I(F(t, \cdot))^{\frac{dq{\frac{r}{r-1}}}{2p(q-{\frac{r}{r-1}})}} \mathd t \)^{\frac{q-{\frac{r}{r-1}}}{q}},\nonumber
	\end{align}
where the constant $C_{p,r,d,T}$ is from applying  Lemma \ref{lemma:sobo}.

The conditions $d/p+2/q+2/r\leq 2$ and $1/q+1/r< 1$ imply $$\frac{dq{\frac{r}{r-1}}}{2p(q-{\frac{r}{r-1}})}= \frac{\frac{d}{p}}{2 \(1-\frac{1}{r}-\frac{1}{q}\)}\leq 1 ,\quad \frac{r}{r-1}<q.$$ Therefore,
 \begin{align*}
 	\int_0^T\int_{\mathbb{R}^{d}}|\tilde{K}(t, x)|^{\frac{r}{r-1}} F(t, x)\mathd x \dif t \leq C_{p,r,d,T}	\|\tilde{K}\|_{L^p_q}^{\frac{r}{r-1}}
 	\(\int_0^T I(F(t,\cdot))^{\alpha_1} \mathd t \)^{\alpha_2},
 \end{align*} with $0<\alpha_1\leq 1$ and $0<\alpha_2<1$.  The result is then concluded by Young's inequality.
\end{proof}

\section{Uniform Fisher information}\label{sec:uniform}
Uniform Fisher information for $N$-particles  system is quite useful when the interaction is singular, see for instance applications in  \cite{fournier2014propagation} on the Biot-Savart law   and \cite{fournier2016propagationlandau} on the homogenous Landau equation with moderate soft potential. The key observation   is that Fisher information provides Sobolev regularities, see Lemma \ref{lemma:sobo}, and controls the singularity of interaction, see Lemma \ref{lemma:lpqf}.
In this section, we derive uniform Fisher information of the joint laws $\{F^N,N \in \mathbb{N}\}$. As mentioned in the introduction, the difficulty is the lack of symmetry.  However,  we are  fortunate enough to establish  the following estimate for the average of singular interactions, which will be applied to derive the main estimate Proposition \ref{prop:entropy} and to identify the limits in the subsequent section.

\begin{lemma}\label{lemma:sym} Assume that the  function $f :[0,T]\times
	\mathbb{R}^d\rightarrow \mathbb{R}^{+}$ satisfies the following property 	
	\begin{align}
		\int_0^T \int_{\mathbb{R}^d}  f ( t,x ) F (t,\dif x)\mathd t\leq    & \alpha + \beta \int_0^TI (F(t,\cdot))\mathd t  ,  \nonumber
	\end{align}
for  some constants $ \alpha,\beta>0$ and any probability measures $F(t, \cdot )$ on $\mathbb{R}^d$ with finite Fisher information for a.e. $t \in [0, T]$.

	Then given  $F^N(t,\cdot)$ the joint distribution of $(X_i(t))$, one has the estimate
	\begin{equation} \label{esti:fishersingular}
		\frac{1}{N^2} \sum_{i \neq j} \int_0^T \mathbb{E} \[f \( t,\frac{1}{\sqrt{2}} (X_i (t)- X_j(t) )\)\]\mathd t \leq  \alpha+\frac{2\beta}{N}\int_0^T
		I (F^N(t,\cdot))\mathd t.
	\end{equation}
\end{lemma}

Note that the  typical choices of $f$ are of the forms $f(x)=|K(x)|^{\theta}$ as in Lemma \ref{lemma:lpqf}.

\begin{remark} If  the joint distribution $F^N$ is symmetric/exchangeable, then the conclusion in Lemma \ref{lemma:sym} is almost trivial. The novelty of this lemma is that we do not impose the symmetry constraint on $F_N$. It would be an interesting topic to study the tensorized property of entropy and Fisher information without the usual symmetry assumption.
\end{remark}

\begin{remark}
The static version of this lemma holds as well. More precisely, if  the function/kernel above doesn't depend on $t$, i.e. $f:\mathbb{R}^d\rightarrow \mathbb{R}^{+}$, and  
$\int_{\mathbb{R}^d}  f ( x ) F (\dif x)\leq     \alpha + \beta I (F)$  for all  $F\in \mathcal{P}(\mathbb{R}^d)$, then it holds 
\begin{equation} 
	\frac{1}{N^2} \sum_{i \neq j}  \mathbb{E} \[f \( \frac{1}{\sqrt{2}} (X_i - X_j )\)\] \leq  \alpha+\frac{2\beta}{N}
	I (F^N).\nonumber
\end{equation}
\end{remark}

\begin{proof}[Proof of Lemma \ref{lemma:sym}]Since we study the 2-body interactions, the case $N=2$ has nothing different compared to  the exchangeable case (for any $N$) and its proof has been essentially verified in the  proof of Lemma 3.3 in \cite{fournier2014propagation}.The only difference here is that we now write an abstract function $f$, instead of a particular form $f(x) = 1/|x|^{\theta}$ as in \cite{fournier2014propagation}. 
	
	 Now we consider the general case  $N\geq 3$,  where the proof is quite non-trivial   and  the key point is the following novel decomposition strategy for some average statistics (dating back to \cite{hoeffding1994probability}).
	
	We start with rewriting the right-hand side of  Eq. \eqref{esti:fishersingular}. Let $\sigma$ be a partition which divides the set $\{ 1, \ldots, N \}$ into $\frac{N}{2}$ groups of pairs of distinct numbers  when $N$ is even,  or $ \frac{N + 1}{2}$  groups with one group containing  a single number and the other $\frac{N-1}{2}$ groups consisting  pairs of distinct numbers when $N$ is odd.  Denote the collection of such
	partitions by $S_N$. Indeed,  when $N$ is even, then up to a permutation,  any  partition in $S_N$ can be reduced to the following canonical form
	\begin{equation*}
		\Big\{ (1, 2), (3, 4), \cdots, (N-1, N)  \Big\}.
	\end{equation*}
	When $N $ is odd, then again any partition in $S_N$ can be reduce to the canonical form
	\begin{equation*}
		\Big \{ (1, 2), (3, 4), \cdots, (N-2, N-1), \{N\} \Big \}.
	\end{equation*}
	Note that we keep the order in those pairs in $\sigma \in S_N$, i.e. when we write that $(i, j) \in \sigma$, by default we mean $i < j$.
	
	Given non-negative variables $(x_{i, j})_{i \ne j}$, one has
	\begin{equation}
		\sum_{i \neq j} x_{i, j} = \sum_{i >j} x_{i, j} + \sum_{i < j} x_{i, j}.
	\end{equation}	
	Below we focus on the summation of $i<j$,   the case   $i>j$ can be dealt in the same manner. When $x_{i,j}=x_{j,i}$, the two summations are  identical. We further find
	\begin{align}\label{fisher-mid1}
		\sum_{i < j} x_{i, j}	=  \frac{1}{ | S_{N - 2} |} \sum_{i < j} | S_{N - 2} | x_{i, j}
		=   \frac{1}{| S_{N-2} |} \sum_{\sigma \in S_N}
		\sum_{ (i, j) \in \sigma} x_{i, j},
	\end{align}
	where  the last equality follows by the fact that  for each pair $(i, j)$ with $i <j$,   it appears  exactly at  $| S_{N - 2} |$ times in the summation $\sum_{\sigma \in S_N}$.  This is  more evident by regarding $\{x_{i,j}\}$ as variables and comparing  the coefficient of each  $x_{i,j}$.
	
	Furthermore,   when counting $|S_N|$, the cardinality of  $S_N$, one can proceed by first arranging a number $j$ for the number $1$ to get a pair $(1, j)$, then there are $|S_{N-2}|$ possible ways to do the partition of the remaining numbers $\{1, 2, \cdots, N\} \setminus \{1, j\}$ when $N$ is even.  This reasoning gives the
	conclusion that
	\begin{equation} \label{fisher-mid2}
		|S_N|=(N-1)|S_{N-2}|, \quad N= 4,6, \cdots,
	\end{equation}
When $N$ is odd, we have $	|S_N|=(N-1)|S_{N-2}| + |S_{N-1}| $, since now we can first arrange a pair like $(1,2)$ or simply the single one $\{ 1\}$. By induction, one has  for even $N$, $|S_{N}| = |S_{N-1}|$.  Consequently,
\begin{equation} \label{fisher-mid22}
 |S_{N}| = N |S_{N-2}|, \quad N = 3,  5, \cdots.
  \end{equation}

	Combining Eq. \eqref{fisher-mid1},  \eqref{fisher-mid2} and \eqref{fisher-mid22}, one obtains that
	\begin{equation}
		\begin{split}
			& \frac{1}{N^2} \sum_{i <j} x_{i, j}= \frac{1}{|S_N|} \sum_{\sigma \in S_N} \frac{|S_N|}{N^2 |S_{N-2}|} \sum_{(i, j) \in \sigma }	 x_{i, j} \\
			& \leq  \frac{1}{|S_N|} \sum_{\sigma \in S_N}   \frac{1}{N} \sum_{(i, j) \in \sigma} x_{i, j}.
		\end{split}
	\end{equation}
	This holds as well for the summation of  ${i>j}$ pairs. Now  letting  $\int_0^T\mathbb{E} [f ( \frac{1}{\sqrt{2}} (X_i - X_j ))]$ play the role of $x_{i, j}$, it thus suffices to show for every partition $\sigma$,
	
	\begin{align}
		\frac{1}{N}  \sum_{ (i, j)  \in \sigma}
\int_0^T\mathbb{E} \Big[f \Big(t, \frac{1}{\sqrt{2}}(X_i (t)- X_j(t) )\Big)\Big] \mathd t  \leqslant &  \frac{\alpha}{2}+\frac{\beta}{N}  \int_0^TI (F^N(t,\cdot))\mathd t . \nonumber
	\end{align}
	
	To this end, for each partition $\sigma \in S_N$, we define $(Y_i)_{1 \leq i \leq N }$ as
	\begin{equation}\label{fisher-mid3}
		\begin{split}
			& Y_i \assign \frac{1}{\sqrt{2}} (X_i - X_j), \quad   Y_j \assign
			\frac{1}{\sqrt{2}} (X_i + X_j), \quad \tmop{for\,  \, } (i, j)
			\in \sigma\,  (\mbox{with }  i< j )  ;  \\
			& Y_i \assign X_i, \quad \tmop{ for\,  } \{ i\} \in \sigma .
		\end{split}
	\end{equation}
	Indeed, without loss of generality, one can always  reduce all $\sigma$ to the canonical form by a permutation of $N$ indices,  for instance in the following let us assume that $N$ is even and   $\sigma = \{ (1, 2), (3, 4), \ldots, (N - 1, N) \}$.  We denote $Y^N = (Y_1, \cdots, Y_N) $ as a function of $X^N = (X_1, \cdots, X_N)$, or simply $Y^N = \Phi(X^N)$,  according to the definition in Eq. \eqref{fisher-mid3}.
	
	Consequently, by change of variables,
	\begin{equation}\label{fisher-mid4}
		\frac{1}{N}  \sum_{ (i, j)  \in \sigma} \int_0^T\mathbb{E} \Big[f \Big(t, \frac{1}{\sqrt{2}}(X_i (t)- X_j(t) )\Big)\Big] \mathd t   = \frac{1}{N} \sum_{k=1}^{N/2} \int_0^T  \mathbb{E} \[f (t,Y_{2k-1}(t))\] \mathd t.
	\end{equation}
	Denote that $\bar F^N = F^N \circ \Phi^{-1}$. Then $\bar F^N$ is nothing but the law of the random variable $Y^N$, and in particular $I(\bar F^N) = I(F^N)$  since  $\Phi$ is an orthogonal transformation. Furthermore, let $\bar{F}_i$ be the
	distribution of $Y_i$. Then recalling our  assumption on the function $f$ and applying   Lemma \ref{lem:summarginal}, the right-hand side of Eq. \eqref{fisher-mid4}  can be further bounded by
	\begin{equation}	
		\begin{split}	
			& \sum_{k=1}^{N/2} \int_0^T  \mathbb{E} \[f (t,Y_{2k-1}(t))\] \mathd t=  \sum_{k=1}^{N/2}\int_0^T \int_{\mathbb{R}^d} f(t,y)  \bar F_{2k-1}(t,\dif y) \mathd t \\&  \leq    \sum_{k=1}^{N/2} \big(\alpha + \beta \int_0^TI (\bar F_{2k-1}(t,\cdot)) \mathd t \big)  \nonumber
				\\&\leq  \frac{N\alpha}{2}+ \beta \int_0^TI (\bar F^N(t,\cdot) )\mathd t
		= \frac{N\alpha}{2}+\beta \int_0^tI(F^N(t,\cdot))\mathd t.
		\end{split}
	\end{equation}
When $N$ is odd, the bound is  simply replacing $N\alpha/2$ by $(N-1)\alpha/2$.

	This completes the proof.
\end{proof}

Applying Lemma \ref{lemma:sym} with $|\tilde K|^{\frac{r}{r-1}}$  in   Lemma  \ref{lemma:lpqf}  playing  the role of $f$, we arrive at the following result.

\begin{corollary}\label{coro:averagesymme}
	For $d\geq2$,
	consider an $\mathbb{R}^d-$valued function  $\tilde{K}\in L^q([0,T],L^p(\mathbb{R}^d))$ with
	\begin{align*}
		\frac{d}{p}+\frac2q+\frac{2}{r}\leq 2,  \quad \frac{d}{p}+\frac{2}{r}<2,\quad r\in (1,\infty].
	\end{align*}   Then for any $\eps>0$ and  any $F^N\in C([0,T],\mathcal{P}(\mathbb{R}^{dN}))$, we have
	\begin{align}
		\frac{1}{N^2}\sum_{i\neq j}\int_0^T\int_{\mathbb{R}^{d}}|\tilde{K}(t, x_i-x_j)|^{\frac{r}{r-1}} F^N(t, x^N)\mathd x^N \dif t\leq \|\tilde{K}\|_{L^p_q}^{\frac{r}{r-1}} \bigg(C_{\eps,p,q,r,d,T} +\frac{\eps}{N} \int_0^T I(F^N(t, \cdot))\mathd t\bigg).\nonumber
	\end{align}
\end{corollary}

Now we are in the position to show the main estimate of this article. Due to the previously established technical lemmas \ref{lemma:lpqf} and \ref{lemma:sym}, the proof is quite  neat.

\begin{proposition}\label{prop:entropy} Suppose that $(\mK_r)$ and $(\mH)$ hold for some $r\in (1,\infty]$.
   For each $N\in \mathbb{N}$ and $T\geq0$,  there exists an entropy solution to \eqref{eq:IPS}. Furthermore, let $\{w^N\}$ be a  sequence satisfying $(\mW_r)$,    there exists a positive  constant $C_T$ such that for all  $t\in [0,T]$ and $N\in \mathbb{N}$,
  \begin{align}
   &\frac{1}{N}H (F_t^N) +\frac{1}{N}\sum_{i =1}^N\int_{\mathbb{R}^{dN}}\<x_i\>^{\gamma} F_t^N \mathd x^N + \frac{1}{2N} \int^t_0 I (F^N_s) \dif s\nonumber
  	\\\leq & \frac{1}{N}H (F_0^N) +\frac{1}{N}\sum_{i =1}^N\int_{\mathbb{R}^{dN}}\<x_i\>^{\gamma} F_0^N \mathd x^N  +C_T. \label{est:uniformfisher}
  \end{align}
\end{proposition}

\begin{proof}
We start with showing  the a priori estimate  uniformly in $N$.

For any $\varphi \in C^2(\mathbb{R}^{dN})$ vanishing at infinity, applying  It\^o's formula to $\varphi (X^N)$ and taking expectation, we arrive at   the Liouville  equation of $F^N$ as
	\begin{equation}\label{eqt:louvillllll}
		\partial_tF^N= \Delta F^N- \sum_{i=1}^N \div_{x_i}\( F^N \, \frac{1}{N}\sum_{j \neq i } w_j^NK(t,x_i-x_j)\)
	\end{equation}
in the  distributional sense.  We then  do some formal computations that can be made rigorous by  approximating the singular kernel $K$ by smooth functions.
Writing  $K=K_1+K_2$ as in ($\mathbb{K}_{r}$), we have
\begin{align}
	\frac{\ud }{\ud t }H(F^N)=& - I(F^N)+\frac{1}{N}\sum_{i \neq j}\int_{\mathbb{R}^{dN}}\nabla_iF^N\cdot w_j^NK(t,x_i-x_j)\mathd x^N\nonumber
	\\=& - I(F^N)-\frac{1}{N}\sum_{i \neq j}\int_{\mathbb{R}^{dN}}F^N w_j^N\div K_1(t,x_i-x_j)\mathd x^N\nonumber
	\\&+\frac{1}{N}\sum_{i \neq j}\int_{\mathbb{R}^{dN}}\nabla_iF^N\cdot w_j^N K_2(t,x_i-x_j)\mathd x^N\nonumber
	\\\assign&  - I(F^N)+J_1+J_2.\label{eqt:unif}
 \end{align}
In the following we apply  Corollary \ref{coro:averagesymme} to handle the interaction terms $J_1$ and $J_2$.  Recalling that the definition  $\|\cdot\|_{l^r}$ of $w^N$,  where the scaling  $N$ appears, and  applying H\"older's inequality, we  obtain
\begin{align*}
	|J_1|\leq&  N \int_{\mathbb{R}^{dN }} F^N \|w^N\|_{l^r} \bigg( \frac{1}{N^2}\sum_{i\neq j} |\div K_1(x_i-x_j)|^{\frac{r}{r-1}}\bigg)^{\frac{r-1}{r}}  \mathd x ^N
	\\\lesssim  &N \|w^N\|_{l^r} \int_{\mathbb{R}^{dN }} F^N \bigg(1+ \frac{1}{N^2}\sum_{i\neq j} |\div K_1(x_i-x_j)|^{\frac{r}{r-1}}\bigg)  \mathd x ^N
	\\\lesssim & N\|w^N\|_{l^r}+ \|w^N\|_{l^r} \frac{1}{N} \sum_{i\neq j} \int_{\mathbb{R}^{d N }} F^N |\div K_1(x_i-x_j)|^{\frac{r}{r-1}} \mathd x^N.
\end{align*}
By the condition ($\mathbb{K}_{r}$), we are allowed to apply Corollary \ref{coro:averagesymme}  with $\div K_1$ playing  the role of $\tilde{K}$.  That means, there exists a  positive constant independent of $N$ such that
\begin{align*}
	\frac{1}{N} \sum_{i\neq j}\int_0^t \int_{\mathbb{R}^{d N }} F^N |\div K_1(s,x_i-x_j)|^{\frac{r}{r-1}} \mathd x^N\mathd  s\leq  CN+ \frac{1}{8\|w^N\|_{l^r}} \int_0^t I (F^N) \mathd s.
\end{align*}
Therefore,  integrating $|J_1|$ w.r.t. time then gives
\begin{align}
	 \int_0^t|J_1|\mathd s\leq CN + \frac{1}{8} \int_0^t I(F^N )\mathd s .\label{inq:unif1}
\end{align}

The procedure for $K_2$ is similar.  We first  apply the Young's inequality to find
\begin{align*}
	 |J_2|\leq& \frac{1}{N}\sum_{i\neq j} |w_j^N|\left(  \epsilon  \int_{\mathbb{R}^{dN}}\frac{|\nabla_{i}F^N|^2}{F^N} \mathd x^N + C_{\epsilon} \int_{\mathbb{R}^{dN}} F^N \left|K_2(x_i-x_j)   \right|^2 \mathd x ^N  \right)
	 \\\leq& \epsilon\|w^N\|_{l^1}I(F^N) + \frac{C_{\epsilon}}{N}\sum_{i\neq j}|w_j^N |  \int_{\mathbb{R}^{dN}} F^N \left|K_2(x_i-x_j)   \right|^2 \mathd x ^N.
\end{align*}
Similarly, let $|K_2|^2$ play the role of $\tilde{K}$ in Corollary \ref{coro:averagesymme}, there exists a  constant $C'_{\epsilon}>0$, depending on $\epsilon $ only,  such that
\begin{align*}
	\frac{C_{\epsilon}}{N}\sum_{i\neq j}|w_j^N |\int_0^t  \int_{\mathbb{R}^{dN}} F^N \left|K_2(s,x_i-x_j)   \right|^2 \mathd x ^N \mathd s \leq  C'_{\epsilon}N+ \frac{1}{8 } \int_0^tI(F^N) \mathd s .
\end{align*}
Choosing $\epsilon $ such that $\eps\sup_N \|w^N\|_{l^1}$ less than $1/8$, we have
\begin{align}
\int_0^t	|J_2|\mathd s \leq  CN+ \frac{1}{8} \int_0^t I (F^N) \mathd s .\label{inq:unif2}
\end{align}

Combining  \eqref{eqt:unif}, \eqref{inq:unif1},  and \eqref{inq:unif2} then yields that
\begin{align}
	H(F_t^N)-H(F_0^N)\leq&  -\int_0^tI(F_s^N) \mathd s +CN + \frac{1}{4}\int_0^t I(F_s^N)\mathd s\nonumber
	\\\leq&  -\frac{3}{4}\int_0^tI(F_s^N) \mathd s +C_{\varTheta}N.\label{inq:unif5}
\end{align}
Here  the  constant  $C_{\varTheta}$ depends on $\varTheta=\{w^N, K_2, \div K_1,p_1,q_1,p_2,q_2,r,d\}$.

On the other hand, testing $\partial_tF_t^N$ with $\sum_{i}|x_i|^{\gamma}$ gives
\begin{align}
\frac{\mathd }{\mathd t}\sum_{i =1}^N\int_{\mathbb{R}^{dN}}\<x_i\>^{\gamma} F^N \mathd x^N =\sum_{i =1}^N\int_{\mathbb{R}^{dN}}F^N \Delta_i\<x_i\>^{\gamma}  \mathd x^N  + \frac{1}{N}\sum_{i\neq j }\int_{\mathbb{R}^{dN}}F^N \nabla_{i} \< x_i\>^{\gamma } w_j^NK(x_i-x_j)\mathd x^N\nonumber.
\end{align}
Since $\gamma\in (0,1)$, the functions $\Delta\<\cdot\>^{\gamma}$ and $\nabla \<\cdot\>^{\gamma}$ are bounded. This implies
\begin{align*}
	\sum_{i =1}^N\left| \int_{\mathbb{R}^{dN}}F^N \Delta_i\<x_i\>^{\gamma}  \mathd x^N  \right| \leq CN,
\end{align*}
and
\begin{align}
	\frac{1}{N}\sum_{i\neq j }\left| \int_{\mathbb{R}^{dN}}F^N \nabla_{i} \< x_i\>^{\gamma } w_j^NK(x_i-x_j)\mathd x^N\right| \leq  \frac{C}{N}\sum_{i\neq j }|w_j^N|\int_{\mathbb{R}^{dN}}F^N \left| K(x_i-x_j)\right| \mathd x^N,\label{inq:unif3}
\end{align}
with a constant $C$ depending on $\gamma$ only.  One may find    the right hand side  of \eqref{inq:unif3} familiar, which enjoys the same formulation  as  $J_1$ and $J_2$. Similarly, by the condition ($\mathbb{K}_{r}$), we can apply Corollary \ref{coro:averagesymme} with $|K|$ playing the role of $\tilde{K}$,  and obtain
\begin{align}
	\frac{1}{N}\sum_{i\neq j }\int_0^t\left| \int_{\mathbb{R}^{dN}}F^N \nabla_{i} \< x_i\>^{\gamma } w_j^NK(x_i-x_j)\mathd x^N\right| \mathd s\leq   CN+ \frac{1}{4} \int_0^t I (F^N)\mathd s. \label{inq:unif6}
\end{align}
Therefore, we have
\begin{align}
	\sum_{i =1}^N\int_{\mathbb{R}^{dN}}\<x_i\>^{\gamma} F_t^N \mathd x^N \leq \sum_{i =1}^N\int_{\mathbb{R}^{dN}}\<x_i\>^{\gamma} F_0^N \mathd x^N + C_TN+ \frac{1}{4} \int_0^t I (F_s^N)\mathd s.\label{inq:unif4}
\end{align}
Now that  we conclude the uniform estimate  \eqref{est:uniformfisher} by summing up  \eqref{inq:unif5} and \eqref{inq:unif4}.

In the following we prove the  existence of entropy solutions to the particle systems \eqref{eq:IPS}.
We consider the approximating systems to \eqref{eq:IPS} with $K$ in \eqref{eq:IPS} replaced by regularized kernels $\{K_{\eps}\}$. When $p,q<\infty$, we can construct $K_{\eps}:=(K*\rho_\eps) \chi_{1/\eps}$ with $\rho_\eps=\eps^{-d}\rho(\eps^{-1}x)$ being the mollifiers and $\chi_R\in C_c^\infty(\mR^d)$ with $\chi_R=1$ for $|x|\leq R$ and $\chi_R=0$ for $|x|>2R$.
We then have
\begin{align}\label{eq:app}
	\|K_{1,\eps}- K_1\|_{L^{p_1}_{q_1}}+\|\div K_{1,\eps}-\div K_1\|_{L^{p_1}_{q_1}}+ \| K_{2,\eps}-K_2\|_{L^{p_2}_{q_2}}\xrightarrow{\eps\rightarrow 0}0.
\end{align}
When $p_1=\infty$ or $p_2=\infty$, i.e. the bounded case, it  is intuitively less singular in our setting, but  requires  additionally truncations in  the approximating procedure. For instance $p_2=\infty$, we   decompose $K_2$ into $K_21_{|x|\leq R}$ plus the reminder $ K_21_{|x|>R}$.  Thus   $K_21_{|x|\leq R}$ is $L^p$-integrable for any $p>1$, we then proceed
the approximations $\{K_{2,\eps}^{R}\}$ for $K_21_{|x|<R}$ as the case  when $p<\infty$.

Since for the approximating system the coefficients $K^\eps$ are smooth and have compact support, there exist unique solutions $X^{\eps,N}$ to the approximating system. Moreover, the related infinitesimal generator for the approximating system is uniform elliptic and has smooth coefficients, which implies that the law of $X^{\eps,N}$
 has a  smooth  density  $F^{\eps,N}\in C([0,T],C^{\infty}(\mathbb{R}^d))$.  Furthermore, the above computation for \eqref{est:uniformfisher} holds for $F^{\eps,N}$ with $C_T$ independent of $\eps$ and $N$.

       { 
       		To pass the limit $\varepsilon \to 0$ and construct an entropy solution to (1.1), we employ a tightness argument. While the procedure is standard for stochastic equations with singular coefficients and  is similar to arguments  for our main results in Section \ref{sec:mf}, we only provide a sketch of the main steps here to avoid repetition.
       		
       		The argument proceeds in two main steps. First, one establishes the tightness of the sequence of laws of $\{X^{\varepsilon,N}\}_{\varepsilon > 0}$ in $C([0,T], \mathbb{R}^{dN})$.  This relies on obtaining crucial  estimates for $X^{\eps,N}$ as in  \eqref{estimate:uni} below, which hold uniformly in $\eps$ (in contrast to the uniform-in-$N$ estimates required in Section \ref{sec:mf}). This allows us to extract a convergent subsequence. By Skorokhod's representation theorem, we may assume, after possibly passing to a new probability space, that this subsequence converges almost surely to a limit process $X^N$ with law $F^N$. Since  the Boltzmann entropy, the $\gamma$-th moment,  and the Fisher information functionals are lower semicontinuous with respect to the weak convergence, { see  \cite[Lemma 3.1 and 3.5]{hauray2014kac}},   the estimate \eqref{est:uniformfisher} holds uniformly  for $F^N$.

       		Second, we must identify the limit and show it solves \eqref{eq:IPS}. The main difficulty is the convergence of the singular interaction term. Following a standard procedure, we split the difference:
       		$$
       		\mathbf{K}_{\varepsilon}(X^{\varepsilon,N}) - \mathbf{K}(X^N) = \underbrace{(\mathbf{K}_{\delta}(X^{\varepsilon,N}) - \mathbf{K}_{\delta}(X^N))}_{\text{(I)}} + \underbrace{(\mathbf{K}_{\varepsilon}(X^{\varepsilon,N}) - \mathbf{K}_{\delta}(X^{\varepsilon,N}))}_{\text{(II)}} + \underbrace{(\mathbf{K}_{\delta}(X^N) - \mathbf{K}(X^N))}_{\text{(III)}},
       		$$
       		where  $\mathbf{K}_\eps(X^{\eps,N})$ is short for $\sum_jw_j^NK_\eps(X^{\eps}_i-X^{\eps}_j)$, and other    abbreviations  are analogous.
       		For a fixed small $\delta > 0$, the kernel $K_\delta$ is continuous and bounded, thus term (I) converges to zero almost surely. The key is to control the error terms (II) and (III). Using the properties of the kernel approximation \eqref{eq:app} and the uniform Fisher information \eqref{est:uniformfisher}, these terms can be controlled in expectation. For instance, the third term is bounded as:
       		$$
       		\mathbb{E}  \int_0^T \left|\mathbf{K}_{\delta}(X^N_t) - \mathbf{K}(X^N_t) \right| dt \le C \(1+\int_0^TI(F_t^N)\mathd t\) \cdot ||K_\delta - K||_{L^p_q} \xrightarrow{\delta \to 0} 0.
       		$$
       		A similar bound holds for term (II) using the uniform-in-$\varepsilon$ Fisher information. Combining these facts in a standard $\varepsilon-\delta$ argument shows the interaction term converges correctly. Finally, apply L\'evy's characterization theorem to confirm that the law of $X^N$ is a solution to the martingale problem for \eqref{eq:IPS}, and thus   $X^N$ is an entropy solution.}
\end{proof}

\section{Random measures with Sobolev regularity}\label{sec:randommeasure}
In this section, we investigate  a sequence of random measures $g^N$ in order to propagate the regularities. 

 When the systems are exchangeable, every accumulation point of $\{\frac{1}{N}\sum_{i=1}^N\delta_{X_i(t)}\assign\nu_N\}$ enjoys finite Fisher information  once  the normalized  Fisher information of the joint laws $F^N$, i.e. $\frac 1 N  I(F_N)$,  is uniformly bounded, see   \cite[Theorem 5.7]{hauray2014kac}.   		However, the exchangeability plays a crucial role in the  above argument, so it cannot be applied in our setting. In order to propagate the regularity of empirical measures  for non-exchangeable systems, we introduce a sequence of  auxiliary random measures $\{g^N\}$ as follows.

We use the disintegration theorem from \cite[Theorem 5.3.1]{ambrosio2008gradient} to write  the product measure $\dif t\times F_t^N(\dif x_1,\dots,\dif x_N)$ as
\begin{align*}
F_t^N(\mathd x_1,...,\mathd x_N)\dif t=&\dif t\times f_{t}^1(\mathd x_1)f_{t}^2(x_1,\mathd x_2)...f_{t}^N(x_1,...,x_{N-1},\mathd x_N)
	\\\assign&\mathd t \times\Pi_{i=1}^Nf_{t}^i(x^{i-1,N},\mathd x_i),
\end{align*}
where $ f_t^i(x^{i-1,N},\dif x_i)$ is a transition probability kernel from $[0,T]\times \mR^{d(i-1)}$ to $\sB(\mR^d)$, i.e. for every  $A\in \sB(\mR^d)$, $(t,x^{i-1,N})\to f_t^i(x^{i-1,N},A)$ is $\sB([0,T]\times \mR^{(i-1)d})$-measurable and for every $t\in [0,T], x^{i-1,N}\in \mR^{d(i-1)}$, $f_{t}^i(x^{i-1,N},\mathd x_i)$ is a probability on $\mR^d$. 
Furthermore, there exists a zero measure set $\cN\subset [0,T]$ such that for $t\in \cN^c$  
\begin{equation}\label{DisIntegration_New}f_t^i(X^{i-1,N}(t),\mathd x_i)=\mathcal{L}(X_i(t)| X^{i-1,N}(t)) \quad \mP-a.s.,\end{equation}
where $\mathcal{L}(X_i(t)| X^{i-1,N}(t))$ is the conditional probability of $X_i(t)$ w.r.t. the $\sigma$-algebra generated by $X^{i-1,N}(t)$. 

Given a set  of deterministic weights $\{\tilde{w}_i^N,1\leq i\leq N\}$,  we define the  random measures $\{g^N\}$  as
\begin{align}
	g^N(t, \mathd x )\assign 
	\frac{1}{N}\sum_{i=1}^N\tilde{w}_i^Nf_{t}^i(X_t^{i-1,N},\mathd x ),\label{defi:gnmarginal}
\end{align}
where  $X_t^{i-1,N}=(X_1(t),\dots, X_{i-1}(t))$. Since $f_t^i(x^{i-1,N},\dif x_i)$ is a transition probability kernel and $t\to X_t^{i-1,N}$ is continuous a.s., $g^N(t,\dif x)$ is also a transition kernel from $[0,T]$ to $\sB(\mR^d)$ a.s.. 

The main results  in this section are Lemma \ref{lem:close} and  Lemma \ref{lem:fisherpreseve}.   Lemma \ref{lem:close}  tells us  that the sequence $\{g^N\}$  merges with  the sequence of empirical measures $\{\tilde{\mu}_N\}$ as $N \to \infty$. 
We use Lemma \ref{lem:fisherpreseve} to obtain the unform regularities of $g^N$.




\subsection{Weakly merging sequences}\label{subsec:gn}
We use the concept {\em weakly merging}  to describe  how ``close" is $g_t^N$ to $\tilde{\mu}_N(t)$.
\begin{definition}\label{def:4.1}
	Two sequences of finite   measure valued stochastic processes $\{\mu_N(t)\}_{t\in[0,T]}$ and $\{\nu_N(t)\}_{t\in[0,T]}$ on $\mathbb{R}^d$ are called  weakly merging if for each $\varphi\in C_b(\mathbb{R}^d)$,  the sequence of  random variables  $\{\<\varphi,{\mu}_N(t)-\nu_N(t)\>\}$ converges to zero for almost all $(t,\omega)\in [0,T]\times \Omega$.
\end{definition}
 When the sequences are deterministic, Definition \ref{def:4.1} agrees with classical version as  in \cite{d1988merging,bogachev2007measure2,dudley2018real}.   The first result below  shows that  $\{g_t^N\}_{t\in[0,T]}$ and $\{\tilde{\mu}_N(t)\}_{t\in[0,T]}$ defined in Theorem \ref{thm:mfpde}  are weakly merging.
\begin{lemma}\label{lem:close} Given a family $\{\tilde{w}^N, N\in \mathbb{N}\}$ satisfying $(\mathbb{W}_r)$ for some  $r\in (1,\infty]$.
	Then   the sequences of finite measure valued stochastic processes   $\{ \tilde{\mu}_N, N \in \mathbb{N}\}$  and $\{ g^N,N \in \mathbb{N}\}$
	are weakly merging.
\end{lemma}

\begin{proof}
	We start with representing  $\tilde{\mu}_N(t)-g_t^N$ via a  martingale difference sequence for a.e. $t\in [0,T]$.  
{ In order to have  better intuition for the random measures $g_t^N$ defined in \eqref{defi:gnmarginal},  one  view $g_t^N$ as a sequence of conditional expectations of the empirical measure $\mu_t^N$. Specifically, we have:
\begin{equation}
	\tilde{\mu}_N(t) = \frac{1}{N} \sum_{i=1}^N \tilde{w}_i^N \delta_{X_i(t)} \quad \text{vs} \quad g_t^N = \frac{1}{N} \sum_{i=1}^N \tilde{w}_i^N \mathbb{E}[\delta_{X_i(t)} | \mathcal{F}_{i-1}],
\end{equation}
where $\mathcal{F}_{i-1} = \sigma(X_1(t), \dots, X_{i-1}(t))$ is the  $\sigma-$fields generated by $(X_1(t),...X_{i-1}(t))$, where we omit the dependence of $\mathcal{F}_{i}$ on $t$ for simplicity. 

For each bounded Borel measurable function $\varphi $ on $\mathbb{R}^d$, the difference $\langle \varphi, \tilde{\mu}_N(t) \rangle - \langle \varphi, g_t^N \rangle$ is precisely the normalized sum of prediction errors:
\begin{equation}
	\langle \varphi, \tilde{\mu}_N(t) \rangle - \langle \varphi, g_t^N \rangle = \frac{1}{N} \sum_{i=1}^N \underbrace{\tilde{w}_i^N \left( \varphi(X_i(t)) - \mathbb{E}[\varphi(X_i(t)) | \mathcal{F}_{i-1}] \right)}_{:= M_i}.
\end{equation}
As defined, it is clear that the sequence $(M_1, \dots, M_N)$ determines a martingale difference sequence with respect to the filtration $(\mathcal{F}_i)_{i=0}^N$ (with $\mathcal{F}_0$ being the trivial $\sigma$-algebra), in the sense that $\mathbb{E}[M_i | \mathcal{F}_{i-1}] = 0$ and $|M_i|\leq 2|\tilde{w}_i^N|\|\varphi\|_{L^{\infty}}$ for all $i=1, \dots, N$.

We apply a generalized Azuma-Hoeffding inequality from  \cite[Theorem 3.13]{mcdiarmid1998concentration},  which is directly applicable to martingale difference sequences with non-uniform bounds. This gives, for all $N \in \mathbb{N}$ and $\eps>0$

}

	\begin{align}
		\mathbb{P}\left( \left| \<\varphi,\tilde{\mu}_N(t)\>-\<\varphi, g_t^N\>\right|  >\eps\right)\leq 2\exp \left(  -\frac{N^2{\eps}^2}{8\|\varphi\|_{L^{\infty}}^2 \sum_{i=1}^N|\tilde{w}_i^N|^2}\right) .
	\end{align}
		When $r\geq 2$, the fact that  $\|\tilde{w}^N\|_{l^2}\leq \|\tilde{w}^N\|_{l^r}$ gives for $t\in \cN^c$
	\begin{align*}
		\mathbb{P}\left( \left| \<\varphi,\tilde{\mu}_N(t)\>-\<\varphi, g_t^N\>\right|  >\eps\right)\leq 2 \exp \( -CN\eps^2 \).
	\end{align*}
	When  $r\in (1,2)$, we use $\|\tilde{w}^N\|_{l^2}^2\leq \|\tilde{w}^N\|_{l^r}^r \|\tilde{w}^N\|_{l^{\infty}}^{2-r}$ and $\|\tilde{w}^N\|_{l^{\infty}}\lesssim N^{\frac{1}{r}}$ to obtain for $t\in \cN^c$
	\begin{align*}
		\mathbb{P}\left( \left| \<\varphi,\tilde{\mu}_N(t)\>-\<\varphi, g_t^N\>\right|  >\eps\right)\leq 2 \exp \( -CN\|\tilde{w}^N\|_{l^{\infty}}^{r-2}\eps^2 \)\leq 2 \exp \( -CN^{2-\frac{2}{r}}\eps^2 \).
	\end{align*}
We note that the universal constant $C$ is independent of $t$ and $N$, which yields that 
\begin{align*}
	\int_{[0,T]\times \Omega}   1_{\left\lbrace  \left| \<\varphi,\tilde{\mu}_N(t)\>-\<\varphi, g_t^N\>\right|  >\eps\right\rbrace }  \mathd t \times \mathd \mathbb{P}  \leq& \sup_{t\in \cN^c} T\mathbb{P}\left( \left| \<\varphi,\tilde{\mu}_N(t)\>-\<\varphi, g_t^N\>\right|  >\eps\right)
	\\\leq &2 T\exp \( -CN^{\theta_r}\eps^2 \),
\end{align*}
where $\theta_r>0$.
	Therefore, for each  $\varphi$,   the sequence
	\begin{align*}
		\left\{  \<\varphi,\tilde{\mu}_N(\cdot)\>-\<\varphi, g_{\cdot}^N\>,N \in \mathbb{N}\right\}
	\end{align*}
	converges to zero in measure.  Furthermore, by the Borel-Cantelli lemma and
	\begin{align*}
		\sum_{N\geq 1} \int_0^T\mathbb{P}\left( \left| \<\varphi,\tilde{\mu}_N(t)\>-\<\varphi, g_t^N\>\right|  >\eps\right) <\infty,
	\end{align*}
	we conclude that  $ \<\varphi,\tilde{\mu}_N(t)\>-\<\varphi, g_t^N\>$
	converges to zero $\mathd t\times \mathd \mathbb{P}$-almost everywhere for each $\varphi\in C_b( \mathbb{R}^d)$.
\end{proof}
\subsection{Regularity of $g^N$}
In the subsequent lemmas, we shall study the regularity of $g^N$.  Classically, we first justify the absolutely continuity of the random measures.
\begin{lemma}\label{lem:tightgnentropy}
Assume that $(\mH)$, $(\mK_r)$ and $(\mW_r)$ hold for some $r\in (1,\infty]$.	For all  $1\leq i\leq N$, $f_t^i(X_t^{i-1,N}, \mathd x )$ is absolutely continuous with respect to the Lebesgue measure  
for	a.s. $(t,\omega)\in [0,T]\times \Omega$.
Furthermore, we have for $\gamma\in (0,1)$,
	\begin{align*}
		\frac{1}{N}\sum_{i=1}^N\mathbb{E}\left(  \int_{\mathbb{R}^d}|x|^{\gamma}f_t^{i}(X_t^{i-1,N}, x)\mathd x +H(f_t^i(X_t^{i-1,N},\cdot))\right) =   \frac{1}{N}\sum_{i=1}^N\mathbb{E}|X_i(t)|^{\gamma}+\frac{1}{N}H(F_t^N).
	\end{align*}
\end{lemma}
\begin{proof} For each $\theta \in \mathcal{P}(\mathbb{R}^d)$,  the chain rule for relative entropy  \cite[Theorem C.3.1]{dupuis2011ldp} gives that
	\begin{align*}
		H(F_t^N|\theta^{\otimes N})&=\sum_{i = 1}^N \int_{\mathbb{R}^{d(i-1)}}H\( f_{t}^i(x^{i-1,N},\cdot)|\theta\) \Pi_{k=1}^{i-1}f_{t}^{k}(x^{k-1,N},\mathd x_k)
		\\&=\sum_{i = 1}^N \int_{\mathbb{R}^{dN }}H\( f_{t}^i(x^{i-1,N},\cdot)|\theta\)F_t^N(\mathd x^N)
		\\&=		\sum_{i=1}^N\mathbb{E}\[H\(f_{t}^i(X_t^{i-1,N},\cdot)|\theta\)\].
	\end{align*}
	Then we choose $\theta$ to be $Ce^{-|x|^{\gamma}}$,  where $C$ is the  normalizing  constant such that $\|\theta\|_{L^1}=1$.  We thus find
	\begin{align}
		\sum_{i=1}^N\mathbb{E}\[H\(f_{t}^i(X_t^{i-1,N},\cdot)|\theta\)\]= H(F_t^N)-N\log  C+\sum_{i=1}^N \mathbb{E}|X_i(t)|^{\gamma}<\infty,\label{inq:tightgn1}
	\end{align}
	which implies the absolutely continuity of $f_t^i(X_t^{i-1,N},\cdot)$ for each $i$.

	On the other hand, we find
	\begin{align}
		H(f_t^i(X_t^{i-1,N},\cdot)|\theta)=H(f_t^i(X_t^{i-1,N},\cdot)) -\log C + \int_{\mathbb{R}^{d}} |x|^{\gamma} f_t^i(X_t^{i-1,N},x)\mathd x .\label{inq:tightgn2}
	\end{align}
	The proof is thus completed by combining \eqref{inq:tightgn1} and \eqref{inq:tightgn2}.
\end{proof}

From Proposition \ref{prop:entropy} we know $F_t^N$ is absolutely continuous w.r.t. the Lebesgue measure, which combined with Lemma \ref{lem:tightgnentropy} implies the absolute continuity of $f_t^i(x^{i-1,N},\dif x_i)$,  denoted as $f_t^i(x_1,\dots,x_i)\dif x_i$. 

Lemma \ref{lem:tightgnentropy} also implies the absolute continuity of $g^N$.

\begin{corollary}
	For each $N$, $g^N$ has a density, still denoted by $g^N$,   with respect to the to the Lebesgue measure  
	for	a.s. $(t,\omega)\in [0,T]\times \Omega$, 
\end{corollary}

The following lemma ensures that the Fisher information does not increase under the construction of random measures.
\begin{lemma}\label{lem:fisherpreseve}For the conditional distributions $\{f_t^i, 1\leq i\leq N\}$ constructed by disintegration as in \eqref{DisIntegration_New},  it holds  that
	\begin{align}
	\sum_{i=1}^N\mathbb{E}\int_0^TI\(f_t^i(X_t^{i-1,N}, \cdot )\)\mathd t \leq \int_0^T I(F_t^N)\mathd t .\label{inq:fisherpreserve}
	\end{align}
	In particular, if $\tilde{w}_i^N=1$ in \eqref{defi:gnmarginal} for all $1\leq i\leq N$, then
	\begin{equation*}
		\mathbb{E}  \int_0^TI(g_t^N)\mathd t \leq  \frac 1 N \int_0^TI(F_t^N)\mathd t .
		\end{equation*}
\end{lemma}

\begin{proof}We first
	rewrite the left side of \eqref{inq:fisherpreserve}  using the definition of the Fisher information, and find
	\begin{align*}
	\sum_{i=1}^N\mathbb{E}\int_0^TI\(f_t^i(X_t^{i-1,N}, \cdot )\)\mathd t=	&\sum_{i=1}^N\mathbb{E}\int_0^T\int_{\mathbb{R}^{d}} \frac{|\nabla_x f_t^i(X_t^{i-1,N},x)|^2}{f_t^i(X_t^{i-1,N},x)}\mathd x \mathd t
		\\=&\sum_{i=1}^N\int_0^T\int_{\mathbb{R}^{dN}}\int_{\mathbb{R}^{d}} \frac{|\nabla_x f_t^i(x_1,...x_{i-1,},x)|^2}{f_t^i(x_1,...,x_{i-1,},x)}\mathd xF_t^N(x^N)\mathd x^N\mathd t .
	\end{align*}
	Using the disintegration of $F_t^N$, we have
	\begin{align}
		\sum_{i=1}^N\mathbb{E}\int_0^TI\(f_t^i(X_t^{i-1,N})\)\mathd t=	&\sum_{i=1}^N\int_0^T\int_{\mathbb{R}^{dN}}\int_{\mathbb{R}^{d}} \frac{|\nabla_{x} f_t^i(x_1,...,x_{i-1},x)|^2}{f_t^i(x_1,...,x_{i-1},x)} \Pi_{j=1}^{N}f_t^j(x_1,...,x_{j})\mathd x\mathd x^N\mathd t\nonumber
		\\=&\sum_{i=1}^N\int_0^T\int_{\mathbb{R}^{di}} \frac{|\nabla_{x_i} f_t^i(x_1,...,x_{i-1},x_i)|^2}{f_t^i(x_1,...,x_{i-1},x_i)} \Pi_{j=1}^{i-1}f_t^j(x_1,...,x_{j})\mathd x_1...\mathd x_i\mathd t\nonumber
		\\=& \sum_{i=1}^N\int_0^T\int_{\mathbb{R}^{dN}} \left| \nabla_{x_i}\log  f_t^i(x_1,...,x_{i-1},x_i)\right| ^2 F_t^N(x^N)\mathd x^N\mathd t,
	\end{align}
where we used $\int_{\mathbb{R}^d} f_t^i(x_1,...,x_{i-1},x_i)\dif x_i=1$.
	On the other hand, we find
	\begin{align*}
		I(F_t^N)	= &\int_{\mathbb{R}^{dN}} \left| \frac{\nabla_{x^N} F_t^N(x^N)}{F_t^N(x^N)}\right|^2 F_t^N(x^N)\mathd x^N \\=&\int_{\mathbb{R}^{dN}} \left|\sum_{i=1}^N\nabla_{x^N} \log f_t^i (x_1,...,x_i)\right|^2 F_t^N(x^N)\mathd x^N
		\\=&\sum_{i=1}^N\int_{\mathbb{R}^{dN}} \left|\nabla_{x^N} \log f_t^i (x_1,...,x_i)\right|^2 F_t^N(x^N)\mathd x^N.
	\end{align*}
	At the last equality we used  the chain rule of Fisher information \cite{zamir1998proof}, equivalent to the fact that  the summation of cross terms equals to zero. More precisely, the summation of all the  cross terms  consists of the following summations with $k\leq i<j$ (we use abbreviated  notations for simplicity) ,
	\begin{align*}
		&\sum_{j>i}	\int_{\mathbb{R}^{dN}} \nabla_{x_k}\log f^i(x_1,...,x_i)\nabla_{x_k} \log f^j (x_1,...,x_i,...,x_j)F^N(x^N)
		\\=& \sum_{j>i}\int_{\mathbb{R}^{d{i}}} \nabla_{x_k}   \log f^i(x_1,...,x_i)  \Pi_{ m\leq i}f^{m}\int_{\mathbb{R}^{d(N-i)}}\frac{\nabla_{x_k}f^j(x_1,...,x_i,...,x_j)}{f^j} \Pi_{l>i}f^{l}
		\\=& \int_{\mathbb{R}^{d{i}}} \nabla_{x_k}   \log f^i(x_1,...,x_i)   \Pi_{ m\leq i}f^{m}\left( \sum_{j>i}\int_{\mathbb{R}^{d(N-i)}}\nabla_{x_k}f^j \Pi_{l>i,l\neq j}f^{l}\right)
		\\=&\int_{\mathbb{R}^{d{i}}} \nabla_{x_k}  \log f^i(x_1,...,x_i)  \Pi_{ m\leq i}f^{m}\left( \nabla_{x_k}\int_{\mathbb{R}^{d(N-i)}} \Pi_{l>i}f^{l}\right)
		\\=&\int_{\mathbb{R}^{d{i}}} \nabla_{x_k}   \log f^i(x_1,...,x_i) \Pi_{ m\leq i}f^{m} \left( \nabla_{x_k}1\right) =0.
	\end{align*}
	Therefore, we arrive at
	\begin{align}
	\sum_{i=1}^N\mathbb{E}\int_0^TI\(f_t^i(X_t^{i-1,N})\)\mathd t =	&\sum_{i=1}^N\int_0^T\int_{\mathbb{R}^{dN}} \left| \nabla_{x_i}\log  f_t^i(x_1,...,x_i)\right| ^2 F_t^N(x^N)\mathd x^N\mathd t\nonumber
		\\\leq &\sum_{i=1}^N\int_0^T\int_{\mathbb{R}^{dN}} \left|\nabla_{x^N} \log f_t^i (x_1,...,x_i)\right|^2 F_t^N(x^N)\mathd x^N\mathd t=\int_0^TI(F_t^N)\mathd t,\nonumber
	\end{align}
	which is exactly \eqref{inq:fisherpreserve}.
	
	When $\tilde{w}_i^N=1$ for all $1\leq i\leq N$,  the convexity of the Fisher information, see \cite[Lemma 3.5]{hauray2014kac}  yields that
	\begin{align}
		I(g^N_t)\leq  \frac{1}{N}\sum_{i=1}^NI\(f_t^i(X_t^{i-1,N}, \cdot )\),\nonumber
	\end{align}
	the result thus follows.
\end{proof}

Since we are interested in non-identical and even unbounded weights, there would be loss of integrability for $g^N$ defined in \eqref{defi:gnmarginal} when $r<\infty$,  due to the appearance of  the weights and the corresponding condition $(\mW_r)$. More precisely,  in Lemma \ref{lem:gnfisher}  below,   the value range of $p$ increases as $r$ increases. 

\begin{lemma}\label{lem:gnfisher}  Given a family $\{\tilde{w}^N, N\in \mathbb{N}\}$ satisfying the condition $(\mathbb{W}_r)$ (i.e. uniformly bounded  in $ l^r$) for some $r\in (1,\infty]$ and assume $(\mH)$, $(\mW_r)$ and $(\mK_r)$, then one has the following results:
	\begin{enumerate}
		\item It holds that
		\begin{align}\label{eq:S1}
			 \mathbb{E} \|g_t^N\|_{L^{\infty}([0,T],L^1(\mathbb{R}^d))}\leq \|w^N\|_{l^1}\leq  \|w^N\|_{l^r}.
		\end{align}
		\item 	For any $1\leq p,q<\infty$ satisfying
		\begin{align}
			\frac{d}{p}+\frac{2(r-1)}{r}\geq d,\quad    \frac{d}{p}+\frac{2}{q}\geq  d ,\label{con:pq1}
		\end{align} it holds that
		\begin{align}\label{eq:S2}
			\mathbb{E}\int_0^T\| g_t^N\|_{L^p}^{q}\mathd t \leq   C \|\tilde{w}^N\|_{l^r}^{q}T+ C \frac{ \|\tilde{w}^N\|_{l^r}^{q} }{N}\mathbb{E}\int_0^T I(F_t^N) \mathd t ,
		\end{align}
	where $p \in [1, \frac{d}{d-2}]$ when $d \geq 3$ and $p \in [1, \infty)$ when $d = 2$.
		\item When $d\geq 3$,  for any $1\leq p,q<\infty$ satisfying 	\begin{align}
			\frac{d}{p}	+\frac{2(r-1)}{r}\geq d+1,\quad   \frac{d}{p}+\frac{2}{q}\geq  d+1,\label{con:pq2},
		\end{align}
		it  holds that
		\begin{align}\label{eq:S3}
			\mathbb{E} \int_0^T \|\nabla g_t^N\|_{L^p}^{q} \mathd t \leq   C\|\tilde{w}^N\|_{l^r}^{q}  T +  C \frac{ \|\tilde{w}^N\|_{l^r}^{q} }{N}\mathbb{E}\int_0^T I(F_t^N) \mathd t,
		\end{align}
	When $d=2$, the result holds for  $p \in [1, 2)$.
	\end{enumerate}
Here the proportional constants are independent of $N$.
\end{lemma}
\begin{proof} The part (1)  follows by the fact that $f^i_t$ is a probability density for each $i$. More precisely,
	\begin{align*}
		 \|g_t^N\|_{L^1(\mathbb{R}^d)}\leq \frac{1}{N}\sum_{i=1}^N |\tilde w_i^N| \int_{\mathbb{R}^{d}} f_t^i(X_t^{i-1,N},x)\mathd x = \| \tilde w^N\|_{l^1}.
	\end{align*}
The proof  of other parts   is based on Lemma \ref{lemma:sobo}, Lemma \ref{lem:fisherpreseve}, and repeatedly applying H\"older's inequality.
	
	For the  part  $(2)$, by Jensen's inequality and H\"older's inequality , we  find
	\begin{align*}
		\mathbb{E}\int_0^T\| g_t^N\|_{L^p}^{q}\mathd t\leq & \mathbb{E} \int_0^T \bigg( \frac{1}{N}\sum_{i=1}^N |\tilde{w}_i^N| \|f_t^i(X_t^{i-1,N},\cdot)\|_{L^p}\bigg)^q\mathd t
		\\\leq &\mathbb{E} \int_0^T\bigg[ \|\tilde{w}^N\|_{l^r} \bigg( \frac{1}{N}\sum_{i=1}^N  \|f_t^i(X_t^{i-1,N},\cdot)\|_{L^p}^{\frac{r}{r-1}}\bigg)^\frac{(r-1)}{r}\bigg]^q \mathd t .
	\end{align*}
	Then applying the first point of   Lemma \ref{lemma:sobo} gives that
	\begin{align*}
		\mathbb{E}\int_0^T\| g_t^N\|_{L^p}^{q}\mathd t\leq & C \|\tilde{w}^N\|_{l^r}^q  \mathbb{E} \int_0^T\bigg( \frac{1}{N}\sum_{i=1}^N  I \(f_t^i(X_t^{i-1,N},\cdot)\)^{\frac{d}{2}(1-\frac{1}{p})\frac{r}{r-1}}\bigg)^\frac{q(r-1)}{r}\mathd t
		\\\leq &  C \|\tilde{w}^N\|_{l^r}^q \mathbb{E} \int_0^T \bigg(1+ \frac{1}{N}\sum_{i=1}^N  I \(f_t^i(X_t^{i-1,N},\cdot)\)^{\frac{d}{2}(1-\frac{1}{p})\max\{ q,\frac{r}{r-1}\}}\bigg)\mathd t ,
	\end{align*}
	where the constant term  $``1"$  comes from applying Young's inequality and appears only when $q< r/(r-1)$. Observe that the condition \eqref{con:pq1} is equivalent to
	\begin{align*}
		\frac{d}{2}\(1-\frac{1}{p}\)\max\{ q,\frac{r}{r-1}\}\leq 1,
	\end{align*}
	thus the estimate \eqref{eq:S2} is concluded by Lemma \ref{lem:fisherpreseve}.
	
	The  proof of Part (3) is almost the same, except that we use the 2nd part  of Lemma \ref{lemma:sobo} to control $\nabla g_t^N$ and also $\nabla f_t^i(X_t^{i-1,N},\cdot)$, instead of the first one. In this case, the condition \eqref{con:pq2} is equivalent to
	\begin{align*}
		\(\frac{d+1}{2}-\frac{d}{2p}\)\max\{ q,\frac{r}{r-1}\}\leq 1.
	\end{align*}
	We omit the rest of the proof to avoid repeating.
\end{proof}
	
\begin{lemma} \label{lem:gregular}Suppose the same setting as in Lemma \ref{lem:gnfisher} and that $g^N_t$  converges to $g_t$ in the space of distributions $\mathcal{S}'(\mathbb{R}^d)$,  i.e. the  dual space of Schwartz functions, $\mathd t \times \mathd \mathbb{P} $ almost everywhere. 	Then the Sobolev regularity estimates \eqref{eq:S1}, \eqref{eq:S2} and \eqref{eq:S3}  hold for $g$.  In particular,  $g_t(\omega) $ has a density w.r.t. the  Lebesgue measure for a.e. $(t,\omega)$.
\end{lemma}
\begin{proof}
	{ We prove the results by considering the variational representation of the $L^p$-norm. We introduce a functional $\Phi(f) := \sup_{\|\varphi\|_{L^{p'}} \le 1;\varphi\in \mathcal{S}(\mathbb{R}^d)} |\langle \varphi,f \rangle|$, which is a well-defined, lower semi-continuous functional on the tempered distributions $S'(\mathbb{R}^d)$, taking values in $[0, +\infty]$. Its value coincides with the $L^p$-norm if $f \in L^p$. The proof proceeds rigorously with this understanding.}
Let $A$ be the set $\{ \varphi\in \mathcal{S}(\mathbb{R}^d), \|\varphi\|_{L^{\frac{p}{p-1}}}\leq 1\}$. By the  convergence    of $g_t^N$ to $g_t$  in $\mathcal{S}'(\mathbb{R}^d)$ for almost every $(t,\omega)$, we find
  \begin{align*}
 \mathbb{E} \int_0^T \|g_t\|_{L^p}^q\mathd t=& \mathbb{E} \int_0^T \left(\sup_{ \varphi\in A}\<\varphi,g_t\>\right)^q\mathd t\leq
     \mathbb{E} \int_0^T \left( \liminf_{N \rightarrow \infty} \sup_{ \varphi\in A}\<\varphi,g_t^N\>\right)^q\mathd t
     \\= &    \mathbb{E} \int_0^T \left( \liminf_{N \rightarrow \infty} \|g_t^N\|_{L^p}\right)^q\mathd t
  \leq \liminf_{N \rightarrow \infty} \mathbb{E} \int_0^T \|g_t^N\|_{L^p}^q\mathd t,
  \end{align*}
where we used the  lower semi-continuity of  supremum and Fatou's lemma. By \eqref{eq:S2}, Proposition \ref{prop:entropy}, and the condition  $(\mathbb{W}_r)$, we conclude that
\begin{align*}
	\mathbb{E} \int_0^T \|g_t\|_{L^p}^q\mathd t
	\leq  C\|\tilde{w}^N\|_{l^r}^{q} T+ C \liminf_{N \rightarrow \infty}\frac{ \|\tilde{w}^N\|_{l^r}^{q} }{N}\mathbb{E}\int_0^T I(F_t^N) \mathd t <\infty.
\end{align*}

Since  $\nabla g_t^N$ converges to $\nabla g_t $ in  $\mathcal{S}'(\mathbb{R}^d;\mathbb{R}^d)$   for  a.e. $(t,\omega) $,  by lower semi-continuity of $\|\cdot\|_{L^p}$-norm and   Fatou's lemma, we derive that the estimate \eqref{eq:S3} holds for the limit $\nabla g_t$.

{ 	The proof of  \eqref{eq:S1} for $g$ follows a similar logic. The $L^1$-norm also admits a dual representation, $\|f\|_{L^1} = \sup_{\|\varphi\|_{L^\infty} \le 1;\varphi\in \mathcal{S}(\mathbb{R}^d)} \langle  \varphi,f \rangle$, which is ensured by \cite[Exercise 4.25]{brezis2010functional}. Therefore, we have
	\begin{align*}
		\mathbb{E}\left[ \mathrm{ess\,sup}_{t \in [0,T]} \|g_t\|_{L^1} \right] &= \mathbb{E}\left[ \mathrm{ess\,sup}_{t \in [0,T]} \sup_{\|\varphi\|_{L^\infty} \le 1;\varphi\in \mathcal{S}(\mathbb{R}^d)} \langle g_t, \varphi \rangle \right] \\
		&\le \mathbb{E}\left[ \liminf_{N\to\infty} \mathrm{ess\,sup}_{t \in [0,T]} \|g_t^N\|_{L^1} \right] \\
		&\le \liminf_{N\to\infty} \mathbb{E}\left[ \mathrm{ess\,sup}_{t \in [0,T]} \|g_t^N\|_{L^1} \right],
	\end{align*}
	where the first inequality follows from the lower semi-continuity of the essential supremum of convex functions, and the second from Fatou's lemma. The right-hand side is uniformly bounded by our estimates in Lemma \ref{lem:gnfisher}.}
\end{proof}
{ 
	\begin{remark}
		In Lemma \ref{lem:gregular} we make the extra assumption that the convergence of the sequence $\{g_t^N\}$ in  $\mathcal{S}'(\mathbb{R}^d)$,   $\mathd t \times \mathd \mathbb{P} $ almost everywhere.  This assumption is satisfied in our setting, as can be seen from the proof of Lemma \ref{lem:limitpoints} below. 
	\end{remark}
 }

\section{Mean-fields limits}\label{sec:mf}
This section is devoted to show the mean-field limits of the interacting system \eqref{eq:IPS}, or more precisely the convergence of empirical measures $\mu_N(t)$ and $\tilde \mu_N(t)$ as in \eqref{def:Empirical},   and  completes the proof of Theorem \ref{thm:mfpde} and Theorem \ref{thm:uniquepde}.
\subsection{Tightness}\label{subsec:tightun}
To apply the classical tightness argument (also known as stochastic compactness method), it requires to find a suitable topology, which is weak enough to show the tightness of laws while sufficiently strong such that the equation (primarily the nonlinear singular  part) as a functional  of solutions is continuous.
 However,  the topology for the convergence and the tightness of the empirical measures on $\mathbb{R}^d$ are too weak to ensure the convergence of the nonlinear term. To solve this problem,
we consider the joint laws of  the random measures $\{g^N\}$ introduced in Section  \ref{sec:randommeasure}, together with the associated weighted empirical measures, to handle  singular interacting kernels. More precisely, we show the tightness of joint laws of  $\{(\tilde{Q}^N, g^N)\}$, where $\tilde{Q}^N$ defined below is the path version of $\tilde{\mu}_N$.

The definition of tightness and the  Prokhorov theorem
is well-known for probability measures. We recall the generalizations for signed measures for convenience, which could be easily found in the textbook \cite{bogachev2007measure2}.
\begin{definition}
	A family $\mathcal{V}$ of Radon measures on a topological space $\mathcal{Y}$ is called  tight if for every $\eps>0$, there exists a compact set $A_{\eps}$ such that $|\nu|(\mathcal{Y} \backslash A_{\eps})<\eps$ for all $\nu\in \mathcal{V}$.
\end{definition}

The following theorem due to Prokhorov connects tightness,  weak convergence sequences, and compactness,  c.f. \cite[Theorems 8.6.2 and 8.6.7]{bogachev2007measure2}.
\begin{lemma}\label{lem:prokh}
Let $\mathcal{Y}$ be a complete separable metric space and let $\mathcal{V}$ be a family of Radon  measures on $\mathcal{Y}$.  Then the following two statements are equivalent:
\begin{enumerate}
	\item every sequence $\{\nu_N\} \subset\mathcal{V}$  contains  a weakly convergent subsequence;
	\item the family   $\mathcal{V}$ is  tight and uniformly bounded in total variation norm.
\end{enumerate}
 Let $\mathcal{V}\subset \mathcal{M} (\mathcal{Y})$ be a uniformly bounded in total  variation norm and tight  family of Radon measures on  $\mathcal{Y}$ . Then  $\mathcal{V}$  has  compact closure in the weak topology.
\end{lemma}

We first show the tightness of laws of the empirical measures on the path space  $C([0,T],\mathbb{R}^d)$, which are defined by
\begin{align}
 \tilde{Q}^N (\cdot) = \frac{1}{N} \sum_i \tilde{w}_i^N\delta_{X_i} \in \cM(C([0,T];\mR^d))  \nonumber.
\end{align}
 Notice that $\tilde{\mu}_N(t)=\tilde{Q}^N\circ \pi_t^{-1}$, where   $\pi_t$ for $t\in [0,T]$ is  the canonical projection from $C([0,T],\mathbb{R}^d)$ to $\mathbb{R}^d$ defined by $\pi_t(X)=X(t)$ for $X\in C([0,T],\mathbb{R}^d)$.
Let  $\phi : C
([0, T], \mathbb{R}^d) \rightarrow [0, \infty]$ be the function

\begin{align}
	\phi (X) : = & \sup_{0\leq s < t\leq T} \frac{| X (t) - X (s) |}{(t - s)^{1 -
			\alpha}} + | X (0) |^{\frac{(r-1)\gamma}{r}},\label{nota:phi1}
\end{align}
where $\gamma\in (0,1)$ and $\alpha\in (\max \{\frac{1}{2}, \frac{d}{2p_1}+ \frac{1}{q_1}, \frac{d}{2p_2}+ \frac{1}{q_2}\},1)$, { the later range  equals to $(\max \{\frac{1}{2}, \frac{d}{2p_1}+ \frac{1}{q_1}\},1)$ under the assumption $(\mK_r)$.} The choice for such $\alpha$ ensures that for $\alpha^*\assign \frac{1}{1-\alpha}$,
\begin{align*}
	2(1-\alpha)<1,\quad 1<\alpha^*, \quad  \frac{d}{p_1}+\frac{2}{q_1}+ \frac{2}{\alpha^*}< 2,\quad \frac{d}{p_2}+\frac{2}{q_2}+ \frac{2}{\alpha^*}< 2, \quad \frac{1}{\alpha}= \frac{\alpha^*}{\alpha^*-1}.
\end{align*}
We will apply Corollary \ref{coro:averagesymme}  with $(\alpha^*,p_1,q_1)$  and  $(\alpha^*,p_2,q_2)$  playing the role of $(r,p,q)$, under the condition ($\mathbb{K}_{r}$) to obtain the following  key uniform estimate.
\begin{lemma}\label{prop:empi} Suppose that $(\mH)$, $(\mW_r)$ and $(\mK_r)$ hold for some $r\in (1,\infty]$. Given a family $\{\tilde{w}^N, N\in \mathbb{N}\}$ satisfying the condition $(\mathbb{W}_r)$, it holds that
	\begin{equation}\label{eqt:nononotight}
\sup_N\mathbb{E}\langle \phi, |\tilde{Q}^N| \rangle =	\sup_N \mathbb{E} \int_{C([0,T],\mathbb{R}^d)}\phi(X) |\tilde{Q}^N|(\mathd X)< \infty.
	\end{equation}

\end{lemma}

\begin{proof}
By the definition of $\phi$, we indeed need to show
	\begin{align}
	 \sup_N \left( \frac{1}{N} \sum_{i = 1}^N |\tilde{w}_i^N|\, \mathbb{E}\,  | X_i (0) |^{\frac{(r-1)\gamma}{r}} +
		\frac{1}{N} \sum_{i = 1}^N|\tilde{w}_i^N| \, \mathbb{E} \sup_{0\leq s < t\leq T} \frac{| X_i
			(t) - X_i (s) |}{(t - s)^{1 - \alpha}} \right) <\infty. \label{estimate:uni}
	\end{align}
The first summation in the bracket concerns on the $\gamma$-th moments of the initial values. Using H\"older's inequality, we find
\begin{align*}
	\frac{1}{N} \sum_{i = 1}^N |\tilde{w}_i^N|\mathbb{E} | X_i (0) |^{\frac{(r-1)\gamma}{r}}\leq& \|\tilde{w}^N\|_{l^r}\bigg( \frac{1}{N}\sum_{i=1}^N  \[\mathbb{E}  | X_i (0) |^{\frac{(r-1)\gamma}{r}} \]^{\frac{r}{r-1}}\bigg)^{\frac{r-1}{r}}
	\\\leq &\|\tilde{w}^N\|_{l^r} \bigg( \frac{1}{N}\sum_{i=1}^N  \mathbb{E}  | X_i (0) |^{\gamma} \bigg)^{\frac{r-1}{r}},
\end{align*}
which is uniformly bounded under   the condition $(\mathbb{H})$.

We then investigate the second summation.  Observe  that
	\begin{align}
		& \frac{1}{N} \sum_{i = 1}^N |\tilde{w}_i^N|\mathbb{E} \sup_{s<t} \frac{| X_i
			(t) - X_i (s) |}{(t - s)^{1 - \alpha}} \nonumber\\
		= & \frac{1}{N} \sum_{i = 1}^N|\tilde{w}_i^N| \mathbb{E} \sup_{s < t} \frac{\left|
			\int^t_{s} \frac{1}{N} \sum_{j\neq i}w_j^N K (\tau ,X_i(\tau) - X_j(\tau))
			\mathd \tau +\sqrt{2} (B_i (t) - B_i (s)) \right|}{(t - s)^{1 - \alpha}} \nonumber\\
		\leqslant & J_1^N + J_2^N  \nonumber
	\end{align}
	where $J_i^N$, $i=1,2$, are defined by
	\begin{align}
		&J_1^N \assign  \frac{1}{N} \sum_{i =1}^N |\tilde{w}_i^N| \mathbb{E} \sup_{s < t}
		\frac{\left| \int^t_{s}\frac{1}{N} \sum_{j\neq i} w_j^NK (\tau,X_i(\tau) - X_j(\tau)) \mathd \tau
			\right|}{(t - s)^{1-\alpha}}, \nonumber
	\\&	J_2^N \assign  \frac{\sqrt{2}}{N} \sum_{i = 1}^N |\tilde{w}_i^N|\mathbb{E} \sup_{s < t}
		\frac{\left| (B_i (t) - B_i (s)) \right|}{(t -
			s)^{1-\alpha}}\nonumber.
	\end{align}
For $J_1^N$ involving the interactions,  we have
	\begin{align}
		J_1^N \leq
		& \frac{1}{N^2} \sum_{i\neq j} |\tilde{w}_i^Nw_j^N|\mathbb{E} \bigg(\sup_{s
			<t} \frac{\int^t_{s}\left| K(\tau,X_i-X_j)\right|   \mathd \tau}{(t -
			s)^{1 - \alpha}} \bigg) \nonumber\\
	\leq 	& \|w^N\|_{l^r}\|\tilde{w}^N\|_{l^r} \bigg( \frac{1}{N^2}\sum_{i \neq j} \bigg[\mathbb{E}   \bigg(\sup_{s
		<t} \frac{\int^t_{s}\left| K(\tau,X_i-X_j)\right|   \mathd \tau}{(t -
		s)^{1 - \alpha}} \bigg)  \bigg]^{\frac{r}{r-1}}  \bigg)^{\frac{r-1}{r}}\nonumber
	\\\leq & \|w^N\|_{l^r}\|\tilde{w}^N\|_{l^r} \bigg( \frac{1}{N^2}\sum_{i \neq j} \bigg[\mathbb{E}  \int^T_{0}\left| K(t,X_i-X_j)\right|^{\frac{1}{\alpha }}   \mathd t   \bigg]^{\frac{\alpha r}{r-1}}  \bigg)^{\frac{r-1}{r}}\nonumber
	\\\leq &  C \|w^N\|_{l^r}\|\tilde{w}^N\|_{l^r} \bigg( T+ \frac{1}{N^2}\sum_{i \neq j} \mathbb{E}  \int^T_{0}\left| K(t,X_i-X_j)\right|^{\max \{\frac{r}{r-1 }, \frac{1}{\alpha}\}}   \mathd t     \bigg)^{\frac{r-1}{r}},\nonumber
	\end{align}
where the constant $T$ is given by   Young's inequality  $|x|^{\alpha r/(r-1)}\leq |x|+1$ when $\alpha r/(r-1)\leq 1$.
We thus obtain
\begin{align*}
	J_1^N \leq & C  \|w^N\|_{l^r}\|\tilde{w}^N\|_{l^r} \bigg(T+ \frac{1}{N^2}\sum_{i \neq j} \mathbb{E}  \int^T_{0}\big(\left| K(t,X_i-X_j)\right|^{\frac{r}{r-1 }}  +\left| K(t,X_i-X_j)\right|^{\frac{1}{\alpha}} \big) \mathd t     \bigg)^{\frac{r-1}{r}}.
\end{align*}
Using Corollary \ref{coro:averagesymme}, we find $J_1^N$ is bounded by the Fisher information.  That is
\begin{align*}
	J_1^N \leq  C  \|w^N\|_{l^r}\|\tilde{w}^N\|_{l^r} \bigg( C+ \frac{1}{N}\int_0^T I(F_t^N)\mathd t \bigg)^{\frac{r-1}{r}},
\end{align*}
for all $N \in \mathbb{N}$. Proposition \ref{prop:entropy} thus implies that  $J_1^N$ is uniformly bounded for all $N \in \mathbb{N}$.

{ Next, the finiteness of $\sup_N J_2^N$ is a consequence of classical results on the sample path regularity of Brownian motion. Specifically, the expectation of the Hölder norm is bounded,
		\begin{align*}
		J_2^N \leq & \sqrt{2}\|\tilde{w}^N\|_{l^1 }    \left[  \mathbb{E} \sup_{s < t}
		\frac{\left| B_1 (t) - B_1 (s)\right|}{(t -
			s)^{1-\alpha}}  \right]<\infty.
	\end{align*}
		provided $1-\alpha < 1/2$. This is a well-known result that can be established using the Kolmogorov continuity criterion (see \cite{zhang2010stochastic,hu2013multiparameter} for further details). The criterion itself can be derived by the general Garsia-Rodemich-Rumsey inequality \cite{garsia1970real}.} The proof of \eqref{eqt:nononotight} is thus completed, and the result follows.
\end{proof}

	\begin{lemma}\label{thm:tight}  Suppose that $(\mH)$, $(\mW_r)$ and $(\mK_r)$ hold for some $r\in (1,\infty]$.
Given a family $\{\tilde{w}^N, N\in \mathbb{N}\}$ satisfying the condition $(\mathbb{W}_r)$,	the laws of the sequence $\{  \tilde{Q}^N,N\in \mathbb{N}\}$
		 are tight on $\mathcal{\mathcal{M}}(C([0,T],\mathbb{R}^d))$.
	\end{lemma}
\begin{proof}
	{ 
	To prove the tightness of the laws of $\{\tilde{Q}^N\}$, we apply a generalized Prokhorov's theorem by identifying a suitable coercive functional.

	Let $\Phi:\mathcal{M}(C([0,T]),\mathbb{R}^d)\rightarrow [0,\infty]$ be the functional defined by
	\begin{align*}
		\Phi(\mu) \assign \left\langle \phi, |\mu|\right\rangle+ \|\mu\|_{TV},
	\end{align*}
	where $\phi$ is a tightness function with precompact level sets (cf. \cite{dupuis2011ldp}). We claim that $\Phi$ itself has precompact level sets. Indeed, for any given $R>0$, the level set $A_R \assign \{\mu \vert \Phi(\mu)\leq R\}$ consists of measures that are uniformly bounded in total variation norm and are tight (as can be deduced from the Chebyshev inequality $|\mu| (\{\phi> c\})\leq \frac{1}{c} \left\langle \phi,|\mu| \right\rangle$). By the generalized Prokhorov's theorem (Lemma \ref{lem:prokh}), the closure of $A_R$ is compact in the weak topology.

	Next, we verify that the expectation of this functional is uniformly bounded for our sequence. By $(\mathbb{W}_r)$ and Lemma \ref{prop:empi}, we have
	\begin{align*}
		\sup_{N \in \mathbb{N}} \mathbb{E} [ \Phi (\tilde{Q}^N)] \le \sup_{N \in \mathbb{N}} \frac{1}{R} \left( \|\tilde{w}^N\|_{l^1}+\mathbb{E} \langle \phi, |\tilde{Q}^N| \rangle \right) < \infty.
	\end{align*}

	Since $\{\mathcal{L}(\tilde{Q}^N)\}$ is a sequence of probability laws on a space where we have found a coercive functional with precompact sublevel sets and a uniform bound on its expectation, the tightness of the sequence follows. }
\end{proof}


The next  result concerns on tightness of laws of $\{g^N\}$.
\begin{lemma}\label{lem:tightini} Suppose that $(\mH)$, $(\mW_r)$ and $(\mK_r)$ hold for some $r\in (1,\infty]$. There exists $p^*>1$ such that
	the laws of $\{g^N,N\in \mathbb{N}\}$ are tight on $L^{p^*}_w([0,T]\times \mathbb{R}^d)$, where $L^{p^*}_w([0,T]\times \mathbb{R}^d)$ denotes $L^{p^*}$-space endowed with the weak topology.
\end{lemma}
\begin{proof}
	By the part $(1)$  of Lemma \ref{lem:gnfisher}, one may choose $p^*>1$ such that $p^*=p=q$ such that
	\begin{align*}
		 \frac{d}{p^*}+ \frac{2(r-1)}{r}\geq d,\quad \frac{d+2}{p^*}\geq d,
	\end{align*}
	which equals to
	\begin{align*}
		1< p^* \leq \min \left\lbrace     \frac{d}{d-\frac{2(r-1)}{r}},\frac{d+2}{d}\right\rbrace,\quad p*<\infty .
	\end{align*}
Thus Lemma \ref{lem:gnfisher} shows there exists such  $p^*>1$ such that
	\begin{align*}
		 \sup_N \mathbb{E} \|g^N\|_{L^{p^*}([0,T]\times \mathbb{R}^d)}^{p^*} \leq C+ C\,   \bigg( \sup_N \|\tilde{w}^N\|_{l^r}^{p^*}\bigg) \bigg( \sup_N \frac{1}{N}\mathbb{E} \int_0^T I(F_t^N )\mathd t\bigg)< \infty.
	\end{align*}
This  uniform bound of Fisher information is ensured  by Proposition \ref{prop:entropy}.  { This implies that the sequence of laws $\{\mathcal{L}(g^N)\}$ is concentrated on a bounded subset of $L^{p^*}([0,T]\times \mathbb{R}^d)$.
	Since $L^{p^*}([0,T]\times \mathbb{R}^d)$ is the dual of $L^{(p^*)'}([0,T]\times \mathbb{R}^d)$ and is therefore a reflexive Banach space for $p^* \in (1, \infty)$, the Banach-Alaoglu theorem states that its bounded subsets are precompact in the weak topology.
	Therefore, the sequence of laws $\{\mathcal{L}(g^N)\}$ is supported on a precompact set, which implies that the sequence is tight in the weak topology of $L^{p^*}([0,T]\times \mathbb{R}^d)$. The proof is thus completed.}
\end{proof}

\subsection{Identify the limits}\label{subsec:identify}Now we extract a convergent  subsequence of    $\{(Q^N,v^N,\tilde{Q}^N,g^N)\}$, where $(Q^N,v^N)$ is defined by replacing  $\tilde{w}^N$ in the definition of $(\tilde{Q}^N,g^N)$  by $w^N$,  and   identify the limiting point as a solution to \eqref{eqt:coupled pde}.

Observe that  $\{w^N\}$ is just a specific example of $\{\tilde{w}^N\}$, by Lemma  \ref{thm:tight} and Lemma \ref{lem:tightini},  one may deduce that the sequence of laws of   $\{(Q^N,v^N,\tilde{Q}^N,g^N), N \in \mathbb{N}\}$ is tight on  the space  $\mathcal{X}$ defined by
\begin{align*}
\mathcal{X}\assign 	\mathcal{M}\(C([0,T],\mathbb{R}^d)\)\times L^{p^*}_w([0,T]\times \mathbb{R}^d) \times 	\mathcal{M}\(C([0,T],\mathbb{R}^d)\)\times L^{p^*}_w([0,T]\times \mathbb{R}^d)
\end{align*}
By the generalized Skorokhod Representation Theorem, also known as the Jakubowski Theorem ({ see the original work \cite[Theorem 2]{jakubowski1998skorokhod} and} the formulation in \cite[Theorem 2.7.1]{hofmanova2018spde}), we deduce the following result.
\begin{proposition}\label{prop:skorohod}
	There exists a subsequence of $\{(Q^N,v^N,\tilde{Q}^N,g^N), N \in \mathbb{N}\}$, without relabeling
	for simplicity, and a probability space $(\Omega^*,\mathcal{F}^*,\mathbb{P}^*)$ with $\mathcal{X}$-valued random variables $\{(Q^{*N},v^{*N},\tilde{Q}^{*N},g^{*N})$, $N \in \mathbb{N}\}$ and $(Q, v,\tilde{Q},g)$ such that
	\begin{enumerate}
		\item For each $N$,  the law of $(Q^N,v^N,\tilde{Q}^N,g^N)$ coincides  with the law of $(Q^{*N},v^{*N},\tilde{Q}^{*N},g^{*N})$.
		
		\item  The sequence of random variables $(Q^{*N},v^{*N},\tilde{Q}^{*N},g^{*N})$ converges to  $(Q, v,\tilde{Q},g)$  in $\mathcal{X}$  $\mP^*$-almost surely.
	\end{enumerate}
\end{proposition}
\begin{remark}To apply the Jakubowski  theorem, one needs to check a  topological  property, that is the space $\mathcal{X}$ is countably separated. This property is closely related to submetrizability (or metrizability for compact spaces).    For our case,  the weak topology of the Polish space $L^{p*}$ is clearly countably separated since its dual space is separable.   As to  $\mathcal{M}(C([0,T],\mathbb{R}^d))$, the required topological property follows by   Koumoullis and  Sapounakis \cite[Theorem 4.1]{koumoullis1984two}.
\end{remark}

For simplicity, we omit the superscript $*$  in the following text.

Now we are able to deduce the convergence of  $\{\tilde{\mu}_N\}$ (and the specific sequence $\{\mu_N\}$). Define  $\tilde{\mu}\assign (\tilde{Q}\circ \pi_t^{-1})_{t\in [0,T]}$.

\begin{corollary}\label{lem:tightmun}
	Up to a subsequence, the convergent sequence of the empirical measure processes $\tilde{\mu}_N$ converges to  $\tilde{\mu}$ in $C([0,T],\mathcal{M}(\mathbb{R}^d))$ almost surely.
\end{corollary}
\begin{proof}
	Recall that the canonical projection  $\pi_t$  from $C([0,T],\mathbb{R}^d)$ to $\mathbb{R}^d$. Clearly, $\tilde{\mu}$ belongs to the space  $C([0,T],\mathcal{M}(\mathbb{R}^d))$.  Given any function  $\varphi\in C_b(\mathbb{R}^d)$, it is  straightforward to check  that  the family of  functions $\{\Phi_t, t\in [0,T]\}$ on $C([0,T],\mathbb{R}^d)$, defined by
	$	\Phi_t(X)\assign\varphi(X_t) $ for $X\in C([0,T],\mathbb{R}^d)$,
	is uniformly bounded and pointwise equicontinuous. Therefore,  we have
	\begin{align*}
		 \sup_{t\in[0,T]}   \left|      \int_{\mathbb{R}^{d}} \varphi(x) \( \tilde{\mu}_N(t)(\mathd x )- \tilde{\mu}_t(\mathd x )\) \right|   =&	\sup_{t\in [0,T]} \left| \int_{\mathbb{R}^{d}} \varphi(x)   \( \tilde{Q}^N\circ \pi_t^{-1} (\mathd x )  -\tilde{Q}\circ \pi_t^{-1}(\mathd x )  \)         \right|
		\\=& \sup_{t\in[0,T]}\left|\int_{C([0,T],\mathbb{R}^d)} \Phi_t (X )  \( \tilde{Q}^N(\mathd X )  -\tilde{Q}(\mathd X )  \) \right| \xrightarrow{N\rightarrow \infty}0,
	\end{align*}
where the convergence follows by the convergence of $\tilde{Q}^N$ and applying  \cite[Exercise 8.10.134]{bogachev2007measure2}.
\end{proof}

The next lemma connects the limiting points of weakly merging sequences $\{\tilde{\mu}_N\}$ and $g^N$.
\begin{lemma}\label{lem:limitpoints}The  subsequence $\{g_t^N\}$ converges  to $\tilde{\mu}_t $ in $\mathcal{S}'(\mathbb{R}^d)$ for almost every $(t,\omega)$.  Furthermore, $g_t$ is a density of $\tilde{\mu}_t$ for almost every $(t,\omega)$.
\end{lemma}
\begin{proof}  For a.e. $(t,\omega)\in [0,T]\times \Omega$ and any $\varphi\in C_0(\mathbb{R}^d)$, we have
	\begin{align}
		\left\langle \varphi, \tilde{\mu}_t\right\rangle = & \lim_{N \rightarrow \infty} \left\langle \varphi, \tilde{\mu}_N(t)\right\rangle \nonumber
		\\=& \lim_{N \rightarrow \infty} \left( \left\langle \varphi, \tilde{\mu}_N(t)- g_t^N\right\rangle + \left\langle\varphi ,g_t^N \right\rangle \right) \nonumber
		\\= & \lim_{N \rightarrow \infty}  \left\langle\varphi ,g_t^N \right\rangle.\nonumber
	\end{align}
	The last equality follows by the fact that $\{\tilde{\mu}_N(t)\}$ and $\{g_t^N\}$ are weakly merging, proved in Lemma \ref{lem:close}. { 	This establishes the convergence for each test function $\varphi$ on a set of full measure, $\mathcal{N}_\varphi^c$, that may depend on $\varphi$. To show that this convergence holds on a single null set for all test functions simultaneously, we use the fact that the pre-dual space of $\mathcal{M}(\mathbb{R}^d)$, $C_0(\mathbb{R}^d)$, is separable. Let $\{\varphi_k\}_{k=1}^\infty$ be a countable dense subset of $C_0(\mathbb{R}^d)$. Let $\mathcal{N} = \bigcup_{k=1}^\infty \mathcal{N}_{\varphi_k}$. As a countable union of null sets, $\mathcal{N}$ is a null set. For any $(t,\omega) \notin \mathcal{N}$, the convergence holds for all $\varphi_k$. By a standard density argument, this implies that for any $(t,\omega) \notin \mathcal{N}$, we have
		$$
		\langle \varphi, \tilde{\mu}_t \rangle = \lim_{N\to\infty} \langle \varphi,g_t^N \rangle \quad \text{for all } \varphi \in C_0(\mathbb{R}^d).
		$$ This implies the  first part of the statement. 
	
	On the other hand, from Proposition \ref{prop:skorohod}, we know that the subsequence $\{g^N\}$ converges to a limit $g$ w.r.t. the weak topology of $L^p([0,T] \times \mathbb{R}^d)$. Combining this with the a.e.-in-time convergence established
	above, we can identify the limit $g$.  For any test function $\varphi \in C_c((0,T) \times \mathbb{R}^d)$, 
 we have
	\begin{align*}
\int_0^T \int_{\mathbb{R}^d} \varphi(t,x) g_t(x) \dif x \dif t = \lim_{N\to\infty} \int_0^T \int_{\mathbb{R}^d} \varphi(t,x) g^N_t(x) \dif x \dif t = \int_0^T \langle \varphi(t,\cdot),\tilde{\mu}_t \rangle \dif t,
	\end{align*}
almost surely. This shows that the two finite measures on the product space, $ g(t,x)\dif x\dif t$ and $ \dif\tilde{\mu}_t(x)\dif t$, coincide when tested against a large class of functions. By a standard density argument (using the separability of the test function space), this implies that the two measures are in fact identical, almost surely.} Finally, by the uniqueness of disintegration (we refer to  \cite[Lemma 10.4.3]{bogachev2007measure2} for the result on  signed measures), we conclude that  $g_t(x)\mathd x= \tilde{\mu}_t(\mathd x )$  for almost every $(t,\omega)$.
\end{proof}
In the following, we shall not distinguish $g$ and $\tilde{\mu}$. The previous lemma together with Lemma \ref{lem:gregular} gives the following.
\begin{corollary}\label{coro:regu}
	The Sobolev regularity estimates \eqref{eq:S1}, \eqref{eq:S2} and \eqref{eq:S3}  hold for $g$.
\end{corollary}

Now we are in the position to identify the limiting point $(v,g)$.
\begin{proposition}\label{prop:limit} Suppose that $(\mH)$ and $(\mK_r)$ hold for some $r\in (1,\infty]$.
 Given two sequences $\{\tilde{w}^N, N\in \mathbb{N}\}$ and  $\{{w}^N, N\in \mathbb{N}\}$   satisfying the condition $(\mathbb{W}_r)$, each limiting point  $(v,g)$ obtained from Proposition \ref{prop:skorohod} is a solution to \eqref{eqt:coupled pde} in the sense of Definition \ref{def:pde}.
\end{proposition}
\begin{proof}
Clearly 
\begin{align}
M_t^N(\varphi)=\left\langle \varphi, \tilde{\mu}_N(t) \right\rangle - &    \left\langle  \varphi,\tilde{\mu}_N(0) \right\rangle - \int_0^t \left\langle \Delta \varphi,\tilde{\mu}_N(s)  \right\rangle\mathd s  -\int_0^t \left\langle \nabla \varphi \cdot K*\mu_N(s),\tilde{\mu}_N(s) \right\rangle  \mathd s \label{iden:1}
\end{align}
is a martingale w.r.t. the filtration generated by $\tilde \mu_N$ and $\mu_N$ for $\varphi\in  C^{\infty}_0(\mathbb{R}^d)$. Observe that  the covariance of martingale $M_t^N(\varphi)$ is of $O(\frac{1}{N})$ by the independence of Brownian motions, we thus have
\begin{align}
	M_t^N(\varphi)\rightarrow 0 ,  \quad \mbox{as } \, N \to \infty, \nonumber
\end{align}
in probability. Up to a subsequence, the martingale  converges to zero almost surely.  

 Since $\tilde{\mu}_N$ converges to $g$ (equivalently to $\tilde{\mu}$)  in $C([0,T],\mathcal{M}(\mathbb{R}^d))$, letting $N\to\infty$ on the both sides of the equality \eqref{iden:1}
leads to
\begin{align}
	\left\langle \varphi, g_t\right\rangle = &\lim_{N \rightarrow \infty}	\left\langle \varphi, \tilde{\mu}_N(t) \right\rangle \nonumber
	\\= &    \left\langle  \varphi,g_0\right\rangle + \int_0^t \left\langle \Delta \varphi,g_s  \right\rangle\mathd s  +\lim_{N \rightarrow \infty} \int_0^t \left\langle \nabla \varphi \cdot K*\mu_N(s),\tilde{\mu}_N(s) \right\rangle  \mathd s .\label{iden:2}
\end{align}
It suffices to identify the  limits of the interacting term.  The main difficulty  is the  lack of continuity of the singular interacting  term with respect to the weak topology in $\mathcal{M}(\mathbb{R}^d)$.  Fortunately, with the estimates Proposition \ref{prop:entropy} and  Corollary \ref{coro:regu}, later we shall show
\begin{align}
\lim_{N \rightarrow \infty} \int_0^t \left\langle \nabla \varphi \cdot K*\mu_N(s),\tilde{\mu}_N(s) \right\rangle  \mathd s
=  \int_0^t\langle \nabla \varphi\cdot  K* v_s,g_s\rangle \mathd s . \label{iden:3}
\end{align}

To obtain \eqref{iden:3}, we approximate $K=K_1+K_2$ by $K_{\eps}=K_{1,\eps}+K_{2,\eps}$ as in the proof of Proposition \ref{prop:entropy}, where $K_{1,\eps}$ and $K_{2,\eps}$, $\eps\in (0,1)$,  are  smooth and compactly  supported functions  satisfying
\begin{align*}
 \|K_{1,\eps}-K_1\|_{L^{p_1}_{q_1}}\rightarrow 0;\quad  \|K_{2,\eps}-K_2\|_{L^{p_2}_{q_2}}\rightarrow 0 ,
\end{align*}
for $p_1,p_2<\infty$.
When $p_1=\infty$ or $p_2=\infty$, we first  truncate $K$ by letting $K= K1_{|\cdot|\leq R}+K1_{|\cdot|>R}$, then proceed the regularization on the local term $K1_{|\cdot|\leq R}$.  The term  $K1_{|\cdot|>R}$ is controlled by finite moments of particles and causes no difficulty in singularity,  we ingore  this term in the following.
Therefore, one may divide the singular interacting  term into a continuous functional on $C([0,T],\mathcal{M}(\mathbb{R}^d))$  and a correction. More precisely,
\begin{align}
\int_0^t \left\langle \nabla \varphi \cdot K*\mu_N(s),\tilde{\mu}_N(s) \right\rangle  \mathd s = \int_0^t \left\langle \nabla \varphi \cdot K_{\eps}*\mu_N(s),\tilde{\mu}_N(s) \right\rangle  \mathd s + R_{\eps, \varphi}^{N},\nonumber
\end{align}
where  the correction $R_{\eps,\varphi}^{N}$ is taken as
	\begin{align}
		R_{\eps,\varphi}^{N}=& \int_0^t \left\langle \nabla \varphi \cdot [K_1-K_{1,\eps}]*\mu_N(s),\tilde{\mu}_N(s) \right\rangle  \mathd s + \int_0^t \left\langle \nabla \varphi \cdot [K_2-K_{2,\eps}]*\mu_N(s),\tilde{\mu}_N(s) \right\rangle  \mathd s \nonumber.
	\end{align}
Similarly, the notation  $R_{\eps,\varphi}$  stands for
\begin{align*}
	R_{\eps,\varphi}\assign & \int_0^t\langle \nabla \varphi\cdot  K* v_s,g_s\rangle \mathd s -  \int_0^t\langle \nabla \varphi\cdot  K_{\eps}* v_s,g_s\rangle \mathd s
	\\=& \int_0^t\langle \nabla \varphi\cdot  [K_1-K_{1,\eps}]* v_s,g_s\rangle \mathd s + \int_0^t\langle \nabla \varphi\cdot  [K_2-K_{2,\eps}]* v_s,g_s\rangle \mathd s .
\end{align*}

 We now claim that for each $\varphi \in C_b^2(\mathbb{R}^d)$,
 \begin{align}
 \mathbb{E} \bigg( 	\sup_{t\in [0,T]} \left| R_{\eps,\varphi}(t)\right| \bigg)\xrightarrow{\eps\rightarrow 0}0;\quad \sup_N \mathbb{E} \bigg( 	\sup_{t\in [0,T]} \left| R_{\eps,\varphi}^N(t)\right| \bigg)\xrightarrow{\eps\rightarrow 0}0.\label{iden:4}
 \end{align}
This  uniform convergence of the corrections is the key ingredient to  deduce  \eqref{iden:3}. Indeed,  the approximations for the kernel $K$ implies 
\begin{align*}
&	  \left| \int_0^t \left\langle \nabla \varphi \cdot K*\mu_N(s),\tilde{\mu}_N(s) \right\rangle  \mathd s -
 \int_0^t\langle \nabla \varphi\cdot  K* v_s,g_s\rangle \mathd s\right|
 	\\\leq &  \left| \int_0^t \left\langle \nabla \varphi \cdot K_{\eps}*\mu_N(s),\tilde{\mu}_N(s) \right\rangle  \mathd s -
 	\int_0^t\langle \nabla \varphi\cdot  K_{\eps}* v_s,g_s\rangle \mathd s\right| +\left| R_{\eps,\varphi}^N(t)\right| + \left| R_{\eps,\varphi}(t) \right| .
\end{align*}
{ 		The first absolute value at the second line involves the limit of the bilinear term as $N \to \infty$. It is important to emphasize that the convergence of this integral relies on the specific topology of the space $C([0,T]; \mathcal{M}(\mathbb{R}^d))^2$. By Proposition \ref{prop:skorohod}, we are on a probability space where $(\mu_N, \tilde{\mu}_N) \to (v,g)$ almost surely. This mode of convergence can be understood as a form of ~``strong convergence in time" combined with ~``weak convergence in space".
	
	Specifically, for a fixed sample point (outside a null set), the convergence in $C([0,T]; \mathcal{M}(\mathbb{R}^d))$ implies that for every $s \in [0,T]$, the measures $\mu_N(s)$ and $\tilde{\mu}_N(s)$ converge weakly to $v_s$ and $g_s$ respectively. This ensures the pointwise-in-time convergence of the integrand:
	$$
	\langle \nabla\varphi * K_\varepsilon * \mu_N(s), \tilde{\mu}_N(s) \rangle \to \langle \nabla\varphi * K_\varepsilon * v_s, g_s \rangle, \quad \text{for every } s \in [0,T].
	$$
	Furthermore,  by the dominated convergence theorem, the integral over $[0,t]$ also converges. This justifies why the term vanishes as $N \to \infty$.} Thus for $\eps>0$
\begin{align}
&	 \lim_{N \rightarrow \infty}\mathbb{E} \left| \int_0^t \left\langle \nabla \varphi \cdot K*\mu_N(s),\tilde{\mu}_N(s) \right\rangle  \mathd s -
\int_0^t\langle \nabla \varphi\cdot  K* v_s,g_s\rangle \mathd s\right| \nonumber
\\\leq &\lim_{N \rightarrow \infty}  { \mathbb{E}}\left| \int_0^t \left\langle \nabla \varphi \cdot K_{\eps}*\mu_N(s),\tilde{\mu}_N(s) \right\rangle  \mathd s -
\int_0^t\langle \nabla \varphi\cdot  K_{\eps}* v_s,g_s\rangle \mathd s\right| \nonumber
\\&+\sup_N \mathbb{E} \left| R_{\eps,\varphi}^N(t)\right| + \mathbb{E} \left| R_{\eps,\varphi}(t) \right| \nonumber
\\\leq & \sup_N \mathbb{E} \left| R_{\eps,\varphi}^N(t)\right| + \mathbb{E} \left| R_{\eps,\varphi}(t) \right| .\nonumber
\end{align}
Choosing $\eps$ sufficient small  and applying \eqref{iden:4}, we arrive at \eqref{iden:3}.

	Now   it remains to prove the claim \eqref{iden:4}.
	
	Recall the definition of $R_{\eps,\varphi}$,  we have
	 \begin{align}
	 	\mathbb{E} \bigg( 	\sup_{t\in [0,T]} \left| R_{\eps,\varphi}(t)\right| \bigg) \leq  &\|\nabla\varphi\|_{L^{\infty}}\mathbb{E}\bigg(
	 	\int_0^T \int_{\mathbb{R}^{d}}   \left| [K_1-K_{1,\eps}]* v_t(x)g_t(x)\right| \mathd x \mathd t \bigg)  \nonumber
	 	\\&+\|\nabla\varphi\|_{L^{\infty}}\mathbb{E}\bigg(
	 	\int_0^T \int_{\mathbb{R}^{d}}   \left| [K_2-K_{2,\eps}]* v_t(x)g_t(x)\right| \mathd x \mathd t\bigg)\nonumber
	 \\	\assign& J_1^{\eps}+J_2^{\eps} .\nonumber
	\end{align}
	For $J_1^\eps$, applying Young's inequality for the convolution of two functions and H\"older's inequality gives
	\begin{align}
		J_1^{\eps}\leq & C\mathbb{E}\bigg( \int_0^T \|K_1-K_{1,\eps}\|_{L^{p_1}}\|v_t\|_{L^{1}}\|g_t\|_{L^p}\mathd t\bigg) \quad \quad \nonumber
		\\\leq & C\|K_1-K_{1,\eps}\|_{L^{p_1}_{q_1}}\mathbb{E}   \|g\|_{L^{p}_q},\quad \quad \quad \quad \quad \quad \quad \quad \quad\(\frac{1}{p_1}+ \frac{1}{p}=\frac{1}{q_1}+\frac{1}{q}=1\).\nonumber
	\end{align}
The condition $(\mathbb{K}_r)$  together with  the  relationship between $(p,q)$ and $(p_1,q_1)$ exactly leads to the condition \eqref{con:pq1}, so that we can apply  Corollary \ref{coro:regu} to find $\mathbb{E}   \|g\|_{L^{p}_q}<\infty$.  We thus have
\begin{align*}
	J_1^{\eps}\leq C\|K_1-K_{1,\eps}\|_{L^{p_1}_{q_1}}\xrightarrow{ \eps\rightarrow 0}0.
\end{align*}
Since it does not involve the divergence of $ K_1$,  the computation for $K_1$ applies to the less singular part  $K_2$ as well, with ${p_1,q_1}$ replaced by $(p_2,q_2)$. We  obtain the convergence of $J_2 ^{\eps}$ and arrive at
\begin{align}
	\mathbb{E} \bigg( 	\sup_{t\in [0,T]} \left| R_{\eps,\varphi}(t)\right| \bigg)\xrightarrow{\eps\rightarrow 0}0.\nonumber
\end{align}

The second uniform  convergence  in \eqref{iden:4} is similar. Since     it concerns  on $N$-particles,  the regularity result we shall apply is Proposition \ref{prop:entropy} instead of Corollary \ref{coro:regu}. Again, the technical result Corollary \ref{coro:averagesymme}  will be used  to handle non-exchangeability.   We start with a simple bound for $|R^N_{\eps,\varphi}|$,
	\begin{align*}
		\mathbb{E} \bigg( \sup_{t\in [0,T]}	|R_{\eps,\varphi}^{N}|\bigg) \leq & \|\nabla\varphi\|_{L^{\infty}} \mathbb{E} \bigg( \int_0^T \frac{1}{N^2}\sum_{i \neq j} |\tilde{w}_i^N||w_j^N| \left| \[K_1-K_{1,\eps}\](X_i-X_j)\right|   \mathd t \bigg)
		\\& +\|\nabla\varphi\|_{L^{\infty}} \mathbb{E} \bigg(\int_0^T \frac{1}{N^2}\sum_{i \neq j} |\tilde{w}_i^N||w_j^N| \left| \[K_2-K_{2,\eps}\](X_i-X_j)\right|   \mathd t \bigg)
		\\\assign & J_1^{\eps,N}+ J_2^{\eps,N}.
	\end{align*}
 We only give the details for the bound of  $J_1^{\eps,N }$ explicitly and the required bound for $J_2^{\eps,N }$ follows similarly.  First, applying H\"older's inequality  w.r.t. the sum over $i,j$ leads to
 \begin{align*}
 	J_1^{\eps,N }=&\|\nabla\varphi\|_{L^{\infty}}\int_0^T  \frac{1}{N^2}\sum_{i \neq j} \mathbb{E} \left( |\tilde{w}_i^N||w_j^N| \left| \[K_1-K_{1,\eps}\](X_i-X_j)\right|  \right)  \mathd t
 	\\\leq &C_{\varphi}  \int_0^T   \|w^N\|_{l^r}\|\tilde{w}^N\|_{l^r} \bigg(\frac{1}{N^2}\sum_{i \neq j} \mathbb{E}  \left| \[K_1-K_{1,\eps}\](X_i-X_j)\right|^{\frac{r}{r-1}}   \bigg) ^{\frac{r-1}{r}}\mathd t  .
 \end{align*}
Since the two sequences $\{w^N\}$ and $\{\tilde{w}^N\}$ are uniformly bounded in $l^r$, there exists a universal constant $C>0$ such that
\begin{align*}
	J_1^{\eps,N }\leq C \int_0^T   \bigg(\frac{1}{N^2}\sum_{i \neq j} \mathbb{E}  \left| \[K_1-K_{1,\eps}\](X_i-X_j)\right|^{\frac{r}{r-1}}   \bigg) ^{\frac{r-1}{r}}\mathd t  .
\end{align*}
Applying Corollary \ref{coro:averagesymme} with $|K_1-K_{1,\eps}|$ playing the role of $\tilde{K}$, we get
\begin{align*}
		J_1^{\eps,N }\leq C\|K_1-K_{1,\eps}\|_{L^{p_1}_{q_1}} \int_0^T  \(1+ \frac{1}{N}I(F_t^N)\)\mathd t  .
\end{align*}
By Proposition \ref{prop:entropy}, Assumptions $(\mH)$, $(\mK_r)$  and the convergence of $K_{1,\eps}$ to $K_1$, we conclude that
\begin{align*}
		J_1^{\eps,N }\leq C\|K_1-K_{1,\eps}\|_{L^{p_1}_{q_1}} \xrightarrow{\eps\rightarrow 0}0.
\end{align*}
The claim \eqref{iden:4}  is thus proved.
\end{proof}
\begin{proof}[Proof of Theorem \ref{thm:mfpde}]
	{ The proof is a combination of our main technical results. By Proposition \ref{prop:skorohod}, for any sequence of integers, there exists a subsequence (which we do not relabel) and a probability space on which $Q^N$ (and $\tilde{Q}^N$) converges almost surely  to a random measure ${Q}$ (and $\tilde{Q}$). This procedure allows us to characterize any accumulation point of the original sequence.
		By the continuous mapping established in Corollary \ref{lem:tightmun}, the corresponding weighted empirical measures $\mu^N$ (and $\tilde{\mu}^N$) also converge almost surely to the corresponding limit $\mu$ (and $\tilde{\mu}$), which is the time marginal of ${Q}$ (and $\tilde{Q}^N$). Finally, Proposition \ref{prop:limit} identifies this limit as a solution to the PDE system. The required regularity estimates for the solution, as per Definition \ref{def:pde}, are guaranteed by Corollary \ref{coro:regu}.}
\end{proof}
\subsection{Uniqueness}
In this section, we prove the uniqueness of the mean-field system \eqref{eqt:coupled pde}.  We divide  Theorem \ref{thm:uniquepde}  into the following Theorem \ref{thm:zzz} and Theorem \ref{thm:xxx}.

\begin{theorem}\label{thm:zzz}
There exists a unique solution $(v,g)\in C([0,T],\mathcal{M}(\mathbb{R}^d))^2$ to \eqref{eqt:coupled pde} in the sense of Definition  \ref{def:pde},   if  the kernel 	$K $ belongs to $    L^{q_2}([0,T],L^{p_2}(\mathbb{R}^d))$ with
	\begin{align}
  \frac{d}{p_2}+ \frac{2}{q_2}+\frac{1}{r}\leq 1,\quad  \frac{d}{p_2}+ \frac{2}{q_2}<1.\label{rela:2}
	\end{align}
\end{theorem}
\begin{proof}
	The proof consists of two parts: the uniqueness of the solution $v$ to the first equation in \eqref{eqt:coupled pde} and the uniqueness of the solution $g$ to the second equation \eqref{eqt:coupled pde} in the  sense  of Definition \ref{def:pde}.   Observe that $g$ solves a linear equation depending on $v$, then it is  natural to study the equation of  $v$ first.

	\textbf{Uniqueness of $v$:}
	
	For general $L^{p_2}_{q_2}$-type kernel,
 the proof is through the mild formulation of \eqref{eqt:coupled pde}. We consider the equation for $v$,
 \begin{align*}
 	v_t= \Gamma_{t}*v_0-\int_0^t \nabla \Gamma_{t-s}*\(K*v_sv_s \) \mathd s .
 \end{align*}

 Let $\kappa>0$ be a positive number satisfying
\begin{align}
	 0<\kappa< \min\left\lbrace \frac{1}{p_2},  \frac{2}{d}(1-\frac{1}{r}), \frac{1}{2d}\(\frac{1}{r}+1- [ \frac{d}{p_2}+ \frac{2}{q_2}+\frac{1}{r}] \)\right\rbrace<\frac{1}{d}. \label{rela:0}
\end{align}
The constraint \eqref{rela:2} implies that  it happens either  $ \frac{d}{p_2}+\frac{2}{q_2}+ \frac{1}{r}< 1$ or $r<\infty$, which together with $r>1$ ensures the existence of $\kappa$.

Suppose that there exist two solutions $v^1$ and $v^2$ starting from the same initial data.  Since  $\kappa< \frac{2}{d}(1-\frac{1}{r})$, we deduce from   Definition \ref{def:pde}  that $v^1$ and $v^2$ belong to $L^{\frac{2}{d\kappa}}  ([0,T, L^{\frac{1}{1-\kappa}}(\mathbb{R}^d))$.
 Computing the $L^{\frac{1}{1-\kappa}}$-norm of $v^1-v^2$ then  leads to
\begin{align}
	\|v_t^1-v^2_t\|_{L^{\frac{1}{1-\kappa}}}\leq &\int_0^t  \left\| \nabla \Gamma_{t-s} * \( K*v_s^1 v_s^1-K*v_s^2v_s^2\)\right\|_{L^{\frac{1}{1-\kappa}}}\mathd s \nonumber
	\\\leq &\int_0^t  \left\| \nabla \Gamma_{t-s} * \( K*v_s^1 \[v_s^1-v_s^2\]\)\right\| _{L^{\frac{1}{1-\kappa}}}\mathd s \nonumber
	\\&+ \int_0^t  \left\| \nabla \Gamma_{t-s} * \( K* \[v_s^1-v_s^2\]v_s^2\)\right\| _{L^{\frac{1}{1-\kappa}}}\mathd s,\label{ine:unique111}
\end{align}
where $\Gamma_{t}$ is the the heat kernel of $\Delta$.
Using Young's convolution inequality,   we have
\begin{align*}
&\int_0^t  \left\| \nabla \Gamma_{t-s} * \( K*v_s^1 \[v_s^1-v_s^2\]\)\right\| _{L^{\frac{1}{1-\kappa}}}\mathd s
\\\leq & \int_0^t \left\|  \nabla \Gamma_{t-s}\right\|_{L^{\frac{1}{1-\kappa^2}}} \left\|   K*v_s^1 \[v_s^1-v_s^2\]\right\|_{L^{\frac{1}{1-\kappa(1-\kappa)}}} \mathd s .
\end{align*}
Furthermore,  by  the property of heat kernel $\|\nabla \Gamma_{t}\|_{L^q}\lesssim t^{\frac{d}{2q}-\frac{d+1}{2}}$ for $q\geq 1$, taking $q= \frac{1}{1-\kappa^2}$ gives
\begin{align}
&\int_0^t  \left\| \nabla \Gamma_{t-s} * \( K*v_s^1 \[v_s^1-v_s^2\]\)\right\| _{L^{\frac{1}{1-\kappa}}}\mathd s \nonumber
\\\lesssim &\int_0^t (t-s)^{\frac{d}{2}(1-\kappa^2)- \frac{d+1}{2}}\left\|  K* v_s^1\right\|_{L^{\frac{1}{\kappa^2}}}\|v^1_s-v^2_s\|_{L^{\frac{1}{1-\kappa}}} \mathd s \quad\quad \quad \quad \text{   }\text{  } \text{ }\nonumber
 \\\lesssim&  \int_0^t (t-s)^{- \frac{1+d\kappa^2}{2}}  \left\|K \right\|_{L^{p_2}} \left\|   v_s^1\right\|_{L^{p_3}}\|v^1_s-v^2_s\|_{L^{\frac{1}{1-\kappa}}} \mathd s ,\quad \quad \quad  \frac{1}{p_3}=1+\kappa^2-\frac{1}{p_2}.\label{ppqq:1}
\end{align}
In the Definition \ref{def:pde}, the maximal spatial integrability we obtained for $v_t^1$ and $v_t^2$ is $L^p(\mathbb{R}^d)$ with $p=\frac{d}{d-2+\frac{2}{r}}<\infty$; this excludes the case $d=2$ and $r=\infty$,    where  the range of $p$ is $ [1,\infty)$.
 Now we check that $p_3\in (1,p)$.  On  one hand,  $p_3>1$ follows by
\begin{align*}
	1+\kappa^2-\frac{1}{p_2}<1+\kappa-\frac{1}{p_2}<1,
\end{align*}
where we used  \eqref{rela:0}.  On the other hand, we find the upper bound $p_3<p$ by noticing
\begin{align*}
1+\kappa^2-\frac{1}{p_2}\geq 1+\kappa^2-\frac{1}{d}(1-\frac{1}{r})=  1- \frac{2}{d}(1-\frac{1}{r})+ \kappa^2 +\frac{1}{d}(1-\frac{1}{r})>\frac{1}{p},
\end{align*}
where the first inequality follows by \eqref{rela:2}, while the last inequality is given by  $\frac{1}{p}=1- \frac{2}{d}(1-\frac{1}{r})$.

By Corollary \ref{coro:regu}, we can take $q_3 >1$  such that  $\frac{1}{q_3}=\frac{d}{2}(1-\frac{1}{p_3})$.
Let $m\geq1$  such that $\frac{1}{q_3} +\frac{1}{q_2}+\frac{1}{m}+\frac{d\kappa}{2}=1$, we  find
\begin{align*}
	\frac{1}{m}=& 1-\frac{1}{q_3}-\frac{1}{q_2}-\frac{d\kappa}{2}
	\\= &1+\frac{d}{2}(\frac{1}{p_3}-1)-\frac{1}{2}\(\frac{d}{p_2}+\frac{2}{q_2}+  \frac{1}{r}-1 \)+ \frac{d}{2p_2}+\frac{1}{2r}-\frac{1}{2}-\frac{d\kappa}{2}
	\\= & \frac{1}{2}+ \frac{d}{2}\( \frac{1}{p_3}+\frac{1}{p_2}-1\)+\frac{1}{2r}-\frac{1}{2}\( \frac{d}{p_2}+ \frac{2}{q_2}+ \frac{1}{r}-1 \)-\frac{d\kappa}{2}
	\\= & \frac{1}{2}+\frac{d\kappa^2}{2}+\frac{1}{2r}-\frac{1}{2}\(\frac{d}{p_2}+\frac{2}{q_2}+  \frac{1}{r}-1 \)-\frac{d\kappa}{2},
\end{align*}
where we used the condition on $(\kappa,p_3,p_2)$ in  \eqref{ppqq:1} to find the last equality.  Recall the condition on $\kappa$, we have
\begin{align}
	\frac{1}{m}>   \frac{1+d\kappa^2}{2} \label{rela:3}.
\end{align}
Applying  H\"older's inequality to \eqref{ppqq:1}, we find
\begin{align}
	&\int_0^t  \left\| \nabla \Gamma_{t-s} * \( K*v_s^1 \[v_s^1-v_s^2\]\)\right\| _{L^{\frac{1}{1-\kappa}}}\mathd s \nonumber
	\\\lesssim  &  \|K\|_{L^{p_2}_{q_2}}\|v^1\|_{L^{p_3}_{q_3}}\|v^1-v^2\|_{L^{\frac{1}{1-\kappa}}_{\frac{2}{d\kappa}}}\left(   \int_0^t  (t-s)^{- \frac{1+d\kappa^2}{2}m }\mathd s  \right)^{\frac{1}{m}}.\nonumber
\end{align}
By \eqref{rela:3} and  the regularity estimate in the Definition \ref{def:pde}, we  conclude that
\begin{align}
	\int_0^t  \left\| \nabla \Gamma_{t-s} * \( K*v_s^1 \[v_s^1-v_s^2\]\)\right\| _{L^{\frac{1}{1-\kappa}}}\mathd s \lesssim \|v^1-v^2\|_{L^{\frac{1}{1-\kappa}}_{\frac{2}{d\kappa}}}. \label{ine:unique1}
\end{align}

Similarly, we have
\begin{align}
&	\int_0^t  \left\| \nabla \Gamma_{t-s} * \( K* \[v_s^1-v_s^2\]v_s^2\)\right\| _{L^{\frac{1}{1-\kappa}}}\mathd s\nonumber
\\\lesssim&  \int_0^t (t-s)^{- \frac{1+d\kappa^2}{2}} \left\| 	K* \[v_s^1-v_s^2\]v_s^2   \right\|_{L^{\frac{1}{1-\kappa(1-\kappa)}}} \mathd s  \nonumber
	\\\lesssim& \int_0^t (t-s)^{- \frac{1+d\kappa^2}{2}} \left\| 	K\right\|_{L^{p_2}}          \left\|   v_s^2\right\|_{L^{p_3}}\|v^1_s-v^2_s\|_{L^{\frac{1}{1-\kappa}}} \mathd s \nonumber
	\\\lesssim& \|v^1-v^2\|_{L^{\frac{1}{1-\kappa}}_{\frac{2}{d\kappa}}}. \label{ine:unique11}
\end{align}
Combining \eqref{ine:unique111}-\eqref{ine:unique11}, we arrive at
\begin{align}
		\|v_t^1-v^2_t\|_{L^{\frac{1}{1-\kappa}}}^{\frac{2}{d\kappa}}\lesssim \int_0^t 		\|v_s^1-v^2_s\|_{L^{\frac{1}{1-\kappa}}}^{\frac{2}{d\kappa}}\mathd s .
\end{align}
Therefore we obtain the  $v^1=v^2$ in $L^{\frac{1}{1-\kappa}}(\mathbb{R}^d)$ for all $t\in (0,T]$ by applying Gronwall's inequality.  We then conclude the uniqueness.

\textbf{Uniqueness of $g$:}

Now we consider  the mild formulation of $g$,
\begin{align*}
	g_t= \Gamma_{t}*g_0-\int_0^t \nabla \Gamma_{t-s}*\(K*{v}_sg_s \) \mathd s .
\end{align*}
Observe that this is a linearized version of the equation for $g$.  Similarly,
suppose that there exist  two solutions $g^1$ and $g^2$, then studying  the $L^{{\frac{1}{1-\kappa}}}$-norm of $g^1-g^2$  leads to
\begin{align*}
	\|g_t^1-g^2_t\|_{L^{\frac{1}{1-\kappa}}}\leq \int_0^t  \left\| \nabla \Gamma_{t-s} * \( K*v_s \[ g_s^1-g_s^2\]\)\right\| _{L^{\frac{1}{1-\kappa}}}\mathd s .
\end{align*}
Similar to \eqref{ine:unique1}, we find
\begin{align*}
		\|g_t^1-g^2_t\|_{L^{\frac{1}{1-\kappa}}}\lesssim \left( \int_0^t 	\|g_s^1-g^2_s\|_{L^{\frac{1}{1-\kappa}}}^{\frac{2}{d\kappa}} \mathd s \right)^{\frac{d\kappa}{2 }} ,
\end{align*}the proof is thus completed by applying Gronwall's inequality.
\end{proof}

When $K$ is the Biot-Savart law on dimension two, Theorem \ref{thm:uniquepde} is indeed the uniqueness of solutions to  the passive scalar advented by the  2D Navier-Stokes equation.
\begin{theorem}\label{thm:xxx}
	There exists a unique solution $(v,g)\in C([0,T],\mathcal{M}(\mathbb{R}^2))^2$ to \eqref{eqt:coupled pde} in the sense of Definition \ref{def:pde}  if  the kernel  $K$ is the 2D Biot-Savart law \eqref{eq:K} and $r=3$.
\end{theorem}
\begin{proof}
	When $K$ is the Biot-Savart law, $v$ solves the vorticity formulation of 2D Navier-Stokes equation. For this case, the uniqueness of solutions with  the regularity properties    \cite[(2.6)]{fournier2014propagation} (i.e.  Corollary \ref{coro:regu} with $r=\infty$)  is already obtained in \cite{fournier2014propagation}, using the well-posedness result in the space $C([0,T),L^1(\mathbb{R}^2))\cap C((0,T), L^{\infty}(\mathbb{R}^d))$ from \cite{ben1994global} and the remark  \cite{brezis1994remarks}.  The strategy in \cite{fournier2014propagation} is to  improve the regularity of solutions by the 	DiPerna-Lions' renormalized solution and  the maximal regularity of the heat equation  so that the solution $v$ meets the conditions in \cite{ben1994global} and   \cite{brezis1994remarks}.
	
	The regularity result  Corollary \ref{coro:regu} is  in fact a generalization of \cite[(2.6)]{fournier2014propagation}. In particular,  Corollary \ref{coro:regu}  implies
	\begin{align*}
		 v\in L^{\infty}([0,T], L^1(\mathbb{R}^2))\cap L^{\frac{p}{p-1}}([0,T], L^p(\mathbb{R}^2)),\quad p\leq r \text{ and }1<p<\infty;
	\end{align*}
and
\begin{align*}
	\nabla v\in L^{\frac{2q}{3q-2}}([0,T], L^q(\mathbb{R}^2),\quad q\leq \frac{2r}{r+2}, \text{ }q<2, \text{ and }1\leq q<2.
\end{align*}
Although here we have the extra restrictions $p\leq r$ and $q\leq 2r/(r+2)$, by letting 
	 $r\geq3$, one can  track  the proof \cite[Theorem 2.5]{fournier2014propagation} to get the uniqueness of the solutions to the  vorticity form of the 2D Navier-Stokes equation.

 Therefore, $v$ is the unique solution  in the sense of Definition \ref{def:pde}. Furthermore, the remark \cite{brezis1994remarks} by Brezis shows that  for $L^1$-valued initial data,
\begin{align}
\lim_{t\rightarrow 0}	t\|v_t\|_{L^{\infty}}=0 .\label{ieq:timereg}
\end{align}
We then use $|K(y)|\lesssim |y|^{-1}$ to obtain
\begin{align*}
|K*v_t(x)|\leq& \left| \int_{|y|\leq  \sqrt{\frac{t}{c(t)}}}K(y)v_t(x-y)\mathd y\right| + \left| \int_{|y|> \sqrt{\frac{t}{c(t)}}}K(y)v_t(x-y)\mathd y\right|
\\\leq & \frac{1}{2\pi}\|v_t\|_{L^{\infty}} \int_{|y|\leq  \sqrt{\frac{t}{c(t)}}}\frac{1}{|y|}\mathd y+ \frac{1}{2\pi}\sqrt{\frac{c(t)}{t}} \|v_t\|_{L^{1}}
\\\lesssim& \sqrt{\frac{t}{c(t)}}\|v_t\|_{L^{\infty}} +\sqrt{\frac{c(t)}{t}} ,
\end{align*}
for  $c(t)>0$ and all $x\in \mathbb{R}^2$. Letting  $c(t)= t\|v_t\|_{L^{\infty}}$ and applying \eqref{ieq:timereg}, we arrive at
\begin{align}
	\lim_{t\rightarrow 0} t^{\frac{1}{2}} \|K*v_t\|_{L^{\infty}}\lesssim	\lim_{t\rightarrow 0}\sqrt{ t\|v_t\|_{L^{\infty}}}=0. \label{timereg}
\end{align}

Suppose there exist two solutions  $g^1$ and $g^2$ to the second equation in \eqref{eqt:coupled pde}. By the mild formulations of solutions, we have
\begin{align*}
	\|g_t^1-g^2_t\|_{L^1}\leq&  \int_0^t  \left\| \nabla \Gamma_{t-s} * \( K*v_s \[ g_s^1-g_s^2\]\)\right\| _{L^1}\mathd s
	\\\lesssim&\int_0^t  (t-s)^{-\frac{1}{2}}\|   K*v_s\|_{L^{\infty}}  \left\| g_s^1-g_s^2\right\| _{L^1}\mathd s .
\end{align*}
Using the time regularity result \eqref{timereg}, we find
\begin{align*}
		\|g_t^1-g^2_t\|_{L^1}\lesssim c_0(t) \int_0^t  (t-s)^{-\frac{1}{2}}s^{-\frac{1}{2}}  \left\| g_s^1-g_s^2\right\| _{L^1}\mathd s ,
\end{align*}
where $c_0(t)=\sup_{s\in [0,t]}s^{\frac{1}{2}}\|K*v_s\|_{L^{\infty}}\rightarrow 0$ as $t\rightarrow0$.
 We then deduce $g_t^1=g_t^2$  up to a  short time $t_0>0$ by  Gronwall's inequality of Volterra type, see for instance \cite[Example 2.4]{zhang2010stochastic}. Applying this argument for finite times, we conclude the the uniqueness  for all $t\in [0,T]$.
\end{proof}

\bibliographystyle{alpha}
\bibliography{ldpvortex}
\end{document}